\documentclass[11pt,a4paper]{article}

\usepackage{epsf,epsfig,amsfonts,amsgen,amsmath,amstext,amsbsy,amsopn,amsthm
}
\usepackage{amsmath,times,mathptmx}
\usepackage{amsfonts,amsthm,amssymb}
\usepackage{amsfonts}
\usepackage{graphics}
\usepackage{latexsym,bm}
\usepackage{amsfonts,amsthm,amssymb,bbding}
\usepackage{indentfirst}
\usepackage{graphicx}
\usepackage{color}
\usepackage[colorlinks=true,anchorcolor=blue,filecolor=blue,linkcolor=blue,urlcolor=blue,citecolor=blue]{hyperref}
\usepackage{float}
\usepackage{tikz}
\evensidemargin=\oddsidemargin\topmargin=-1.5cm

\newtheorem{thm}{Theorem}[section]

\newtheorem{lem}{Lemma}[section]

\addtocounter{section}{0}

\begin{document}
\title{Spectral extremal problem on $t$ copies of $\ell$-cycle\footnote{Supported by the National Natural Science Foundation of China (Nos. 12171066, 12271162), Anhui Provincial Natural Science Foundation (No. 2108085MA13) and Major Natural Science Research Project of Universities in Anhui Province (No. 2022AH040151).}}
\author{{\bf Longfei Fang$^{a,b}$},
{\bf Mingqing Zhai$^b$}\thanks{Corresponding author: mqzhai@chzu.edu.cn
(M. Zhai)}, {\bf Huiqiu Lin$^{a}$} \\
\small $^{a}$ School of Mathematics, East China University of Science and Technology, \\
\small  Shanghai 200237, China\\
\small $^{b}$ School of Mathematics and Finance, Chuzhou University, \\
\small  Chuzhou, Anhui 239012, China\\
}

\date{}
\maketitle
{\flushleft\large\bf Abstract}
Denote by $tC_\ell$ the disjoint union of $t$ cycles of length $\ell$.
Let $ex(n,F)$ and $spex(n,F)$ be the maximum size
and spectral radius over all $n$-vertex $F$-free graphs, respectively.
In this paper, we shall pay attention to the study of both $ex(n,tC_\ell)$ and $spex(n,tC_\ell)$.
On the one hand, we determine $ex(n,tC_{2\ell+1})$
and characterize the extremal graph
for any integers $t,\ell$ and
$n\ge f(t,\ell)$, where $f(t,\ell)=O(t\ell^2)$.
This generalizes the result on $ex(n,tC_3)$ of Erd\H{o}s
[Arch. Math. 13 (1962) 222--227] as well as the research
on $ex(n,C_{2\ell+1})$ of F\"{u}redi and Gunderson
[Combin. Probab. Comput. 24 (2015) 641--645].
On the other hand,
we focus on the spectral Tur\'{a}n-type function $spex(n,tC_{\ell})$,
and determine the extremal graph
for any fixed $t,\ell$ and large enough $n$.
Our results not only extend some classic spectral extremal results on triangles,
quadrilaterals and general odd cycles due to Nikiforov,
but also develop the famous spectral even cycle conjecture
proposed by Nikiforov (2010) and confirmed by Cioab\u{a}, Desai and Tait (2022).

\begin{flushleft}
\end{flushleft}

\section{Introduction}

Given a graph $F$, a graph is said to be \textit{$F$-free}
if it does not contain a subgraph isomorphic to $F$.
The \emph{Tur\'{a}n number} of $F$, denoted by $ex(n,F)$,
is the maximum number of edges in an $n$-vertex $F$-free graph.
An $F$-free graph is said to be \textit{extremal} with respect
to $ex(n,F)$, if it has $n$ vertices and $ex(n,F)$ edges.
Denote by $T_{n,r}$ the complete $r$-partite graph on $n$ vertices
in which all parts are as equal in size as possible.
An interesting graph in Tur\'{a}n-type problem is perhaps a cycle.
In 2015, F\"{u}redi and Gunderson \cite{Furedi1} determined $ex(n,C_{2\ell+1})$ for all $n$ and $\ell$, and specially, $T_{n,2}$ is the unique extremal graph when $n\geq4\ell$.
However, up to now the exact value of $ex(n,C_{2\ell})$ is still open.
Given a graph $F$, we denote by $tF$ the disjoint union of $t$ copies of $F$.
The study of the Tur\'{a}n number of $tC_\ell$ can be dated back to 1962, Erd\H{o}s \cite{Erdos1} determined $ex(n,tC_3)$ for $n>400(t-1)^2$,
and characterized the unique extremal graph $K_{t-1}+T_{n-t+1,2}$
(that is, the join of $K_{t-1}$ and $T_{n-t+1,2}$).
Subsequently, Moon \cite{Moon} proved that Erd\H{o}s's result is still valid whenever $n>\frac{9t-11}{2}$.
In addition, Erd\H{o}s and P\'{o}sa \cite{Erdos2} also showed that $ex(n,t\mathcal{C})=(2t-1)(n-t)$ for $t\ge 2$ and $n\ge 24t$,
where $t\mathcal{C}$ is the family of $t$ disjoint cycles.
In this paper, we further determine the Tur\'{a}n number $ex(n,tC_{2\ell+1})$
by the following theorem.
It should be noted that if $n$ is sufficiently large, our result is a special
case of Theorem 2.2 due to Simonovits \cite{Simonovits}.

\begin{thm}\label{theorem1.1}
Let $t,\ell,n$ be three integers with $t,\ell\ge 2$ and $n\ge \Big\lfloor\frac{(8t\ell+4\ell+3t-6)^2}{4\lfloor\frac{t}{2}\rfloor}\Big\rfloor
+8t\ell+4t+4\ell-5$.
Then $K_{t-1}+T_{n-t+1,2}$ is the unique extremal graph with respect to $ex(n,tC_{2\ell+1})$.
\end{thm}

Let $A(G)$ be the adjacency matrix of a graph $G$,
and $\rho(G)$ be its spectral radius, that is,
the maximum modulus of eigenvalues of $A(G)$.
The spectral extremal value of a given graph $F$, denoted by $spex(n,F)$,
is the maximum spectral radius over all $n$-vertex $F$-free graphs.
An $F$-free graph on $n$ vertices with maximum spectral radius is called an \textit{extremal graph} with respect to $spex(n,F)$.
Note that $\rho(G)\geq\frac {2m}n$ for each graph $G$ with $n$ vertices and $m$ edges.
Thus we always have
$ex(n,F)\leq \frac n2spex(n,F).$

In recent years, the investigation on $spex(n,F)$ has become very popular (see \cite{Cioaba2,Cioaba,Desai,Li,LIY,LIN1,LIN2,NIKI,TAIT1,TAIT2,WANG,ZHAI} ).
In this paper, we are interested in studying $spex(n,tF)$ for some given $F$.
Up to now,  $spex(n,tF)$ and its corresponding extremal graphs were studied
for some special cases
(see $spex(n,tK_2)$ \cite{Feng}, $spex(n,tP_\ell)$ \cite{Chen-M}, $spex(n,tS_\ell)$ \cite{Chen-M2}, $spex(n,tK_\ell)$ \cite{Ni}).
In this paper, we consider that $F$ is a cycle of given length.

We first investigate the case that $F$ is an odd cycle.
Note that Nikiforov \cite{Nikiforov2} determined $spex(n,C_{2\ell+1})$
for sufficiently large $n$.
Using Theorem \ref{theorem1.1} and Nikiforov's result on $spex(n,C_{2\ell+1})$,
we prove the following theorem.

\begin{thm}\label{theorem1.3}
For any two given positive integers $t,\ell$ and sufficiently large $n$,
$K_{t-1}+T_{n-t+1,2}$ is the unique extremal graph  with respect to $spex(n,tC_{2\ell+1})$.
\end{thm}

Next, we focus on the case that $F$ is an even cycle.
When $t=1$, it can be reduced to a classic spectral Tur\'{a}n-type problem,
that is, $spex(n,C_{2\ell})$,
which was initially investigated by Nikiforov
\cite {Nikiforov5, Nikiforov1}.
Denote by $S_{n,\ell}$ the join of an $\ell$-clique
with an independent set of size $n-\ell$.
Furthermore, let $S_{n,\ell}^+$ be the graph obtained from $S_{n,\ell}$ by adding
an edge within its independent set,
and $S_{n,\ell}^{++}$ be the graph obtained from $S_{n,\ell}$
by embedding a maximum matching within its independent set.
Nikiforov \cite {Nikiforov5} and Zhai et al. \cite{Zhai-3} determined $spex(n,C_4)$ as well as its unique extremal graph $S_{n,1}^{++}$ for odd and even $n$ respectively.
In 2010,  Nikiforov \cite{Nikiforov1} posed the \textit{spectral even cycle conjecture}
as follows:
$S_{n,\ell-1}^{+}$ is the unique extremal graph with respect to
$spex(n,C_{2\ell})$ for $\ell\geq3$ and $n$ large enough.
This conjecture was completely resolved by Cioab\u{a}, Desai and Tait \cite{Cioaba1} in 2022. In this paper, we develop the conjecture and obtain the following result.

\begin{thm}\label{theorem1.5}
Let $t,\ell$ be given positive integers and $n$ be sufficiently large. Then\\
(i) $S_{n,2t-1}^{++}$ is the unique extremal graph with respect to $spex(n,tC_4)$;\\
(ii) $S_{n,\ell t-1}^+$ is the unique extremal graph with respect to $spex(n,tC_{2\ell})$
for $\ell\ge 3$.
\end{thm}

By the famous Erd\H{o}s-Gallai theorem,
every $C_{2\ell}$-free graph has a property that
there are at most $\ell n$ edges within the neighborhood of each vertex.
However, this does not hold for $tC_{2\ell}$-free graphs provided that $t\geq2$.
Hence, spectral extremal problem on $tC_{2\ell}$-free graphs is different from
that on $C_{2\ell}$-free graphs.
To this end, we use induction to obtain a very important structural
proposition on the spectral extremal graph $G$, precisely,
$G-\{u\}$ contains $t-1$ disjoint $2\ell$-cycles for each vertex $u\in V(G)$.
This presents a key approach to prove Theorem \ref{theorem1.5}.

The remainder of this paper is organized as follows. In Section \ref{section2},
some preliminary results and lemmas are presented.
In Section \ref{section3}, we use Erd\H{o}s-Moon theorem on $ex(n,tC_3)$ and
more structural analysis to prove Theorem \ref{theorem1.1}, that is, obtain the exact Tur\'{a}n number of $tC_{2\ell+1}$ for $n\geq f(t,\ell)$.
In Section \ref{section4}, we use Theorem \ref{theorem1.1}
and a stability method to show Theorem \ref{theorem1.3}.
In Section \ref{section5},
we give the proof of Theorem \ref{theorem1.5} by using induction and the spectral even cycle results in \cite{Cioaba1,Nikiforov5,Zhai-3}.
We try our best in dealing with the induction procedure,
which may be helpful for other spectral Tur\'{a}n-type problems.

\section{Preliminaries}\label{section2}

Given a simple graph $G$,
we use $V(G)$ to denote
the vertex set, $E(G)$ the edge set, $|G|$ the number of vertices, $e(G)$ the number of edges,
$\nu(G)$ the matching number, $\Delta(G)$ the maximum degree, $\delta(G)$ the minimum degree,  respectively.
For a vertex $v\in V(G)$,
we denote by $N_G(v)$ its neighborhood
and $d_G(v)$ its degree in $G$.
Given two disjoint vertex subsets $S$ and $T$.
Let $G[S]$ be the subgraph induced by $S$, $G-S$
be the subgraph induced by $V(G)\setminus S$,
and $G[S,T]$ be the bipartite subgraph on the vertex set $S\cup T$
which consists of all edges with one
endpoint in $S$ and the other in $T$.
For short, we write $e(S)=e(G[S])$ and $e(S,T)=e(G[S,T])$.
Let $K_{n_1,\dots,n_r}$ be the complete $r$-partite graph with classes of sizes $n_1,\dots,n_r$.
If $\sum_{i=1}^{r}n_i=n$ and $|n_i-n_j|\le 1$ for any two integers $i,j\in \{1,\dots,r\}$,
then $K_{n_1,\dots,n_r}$ is exactly the
$n$-vertex $r$-partite Tur\'{a}n graph $T_{n,r}$.
Let $F+H$ be the join and $F\cup H$ be the union, of $F$ and $H$, respectively.
Particularly, we denote by $tF$ the disjoint union of $t$ copies of $F$.

In this section, we introduce some lemmas which will be used in
the proofs of Theorems \ref{theorem1.1}, \ref{theorem1.3} and \ref{theorem1.5}.
The first one is due to Erd\H{o}s \cite{Erdos1} and Moon \cite{Moon}.

\begin{lem}\label{lemma-2.1} \emph{(\cite{Erdos1,Moon})}
Let $t,n$ be two positive integers with $n\ge \lfloor\frac{19t-9}{2}\rfloor$.
Then $$ex(n,tC_3)=\binom{t-1}{2}+(t-1)(n-t+1)+\Big\lfloor\frac{(n-t+1)^2}{4}\Big\rfloor.$$
Furthermore, $K_{t-1}+T_{n-t+1,2}$ is the unique extremal graph with respect to $ex(n,tC_3)$.
\end{lem}

Given two integers $\nu$ and $\Delta$,
define $f(\nu,\Delta)=\max\{e(G)~|~\nu(G)\le \nu, \Delta(G)\le \Delta\}$.
In 1976, Chv\'{a}tal and Hanson \cite{Chvatal} obtained the following result.

\begin{lem} \label{lemma-2.2}\emph{(\cite{Chvatal})}
For every two integers $\nu\ge 1$ and $\Delta\ge 1$, we have
$$f(\nu,\Delta)=\Delta\nu+\Big\lfloor\frac{\Delta}{2}\Big\rfloor
\Big\lfloor\frac{\nu}{\lceil\frac{\Delta}{2}\rceil}\Big\rfloor\le \nu(\Delta+1).$$
\end{lem}

The following stability theorem was given by Nikiforov \cite{Nikiforov4}.

\begin{thm}\label{thm0}\emph{(\cite{Nikiforov4})}
Let $r\ge 2$, $\frac{1}{\ln n}<c<r^{-8(r+21)(r+1)}$, $0<\varepsilon<2^{-36}r^{-24}$ and $G$ be an $n$-vertex graph.
If $\rho(G)>(1-\frac1r-\varepsilon)n$, then one of the following holds:\\
(i) $G$ contains a $K_{r+1}(\lfloor c\ln n\rfloor, \dots,\lfloor c\ln n\rfloor,\lceil n^{1-\sqrt{c}}\rceil)$;\\
(ii) $G$ differs from $T_{n,r}$ in fewer than $(\varepsilon^{\frac{1}{4}}+c^{\frac{1}{8r+8}})n^2$ edges.
\end{thm}

From Theorem \ref{thm0},
Desai et al. \cite{Desai} obtained the following stability result.
Theorem \ref{thm0} and the following lemma present an efficient approach
to study spectral extremal problems.

\begin{lem} \label{lemma-2.3}\emph{(\cite{Desai})}
Let $F$ be a graph with chromatic number $\chi(F)=r+1$.
For every $\varepsilon>0$, there exist $\delta>0$ and $n_0$ such that
if $G$ is an $F$-free graph on $n\ge n_0$ vertices with $\rho(G)\ge (1-\frac1r-\delta)n$,
then $G$ can be obtained from $T_{n,r}$ by adding and deleting at most $\varepsilon n^2$ edges.
\end{lem}

The following spectral extremal result on odd cycles
is due to Nikiforov \cite{Nikiforov2}.

\begin{lem} \label{lemma2.4}\emph{(\!\!\cite{Nikiforov2})}
Let $\ell$ be a given positive integer and $n$ be large enough.
Then, $T_{n,2}$ is the unique extremal graph with respect to $spex(n,C_{2\ell+1})$.
\end{lem}

The following result is known as Erd\H{o}s-Gallai theorem.


\begin{lem} \label{lemma2.5} \emph{(\cite{Erdos3})}
Let $n$ and $\ell$ be two integers with $n\ge \ell\ge 2$.
Then $ex(n,P_\ell)\le\frac{(\ell-2)n}{2}$, with equality if and only if $n=t(\ell-1)$ and $G\cong t K_{\ell-1}$.
\end{lem}

We note that the best current bound for $ex(n,C_{2\ell})$
was given by He \cite{He}, who improved
on a bound $ex(n,C_{2\ell})\leq \big(80\sqrt{\ell}\log\ell+o(1)\big)n^{1+\frac1\ell}$
of Bukh and Jiang \cite{Bukh} by reducing a factor of
$\sqrt{5\log\ell}$. However, for our purposes the dependence of the
multiplicative constant on $\ell$ is not important.
For convenience, we use the following version, which improves a
best known bound of Verstra\"{e}te \cite{Verstraete} by a factor $8+o(1)$.

\begin{lem} \label{lemma2.6} \emph{(\cite{OP})}
For all $\ell\ge 2$ and $n\geq1$, we have
$$ex(n,C_{2\ell})\le (\ell-1)n(n^{\frac{1}{\ell}}+16).$$
\end{lem}

\section{Proof of Theorem \ref{theorem1.1}}\label{section3}

In this section, we give the proof of Theorem \ref{theorem1.1}.
More precisely, we will extend the Tur\'{a}n-type result on disjoint triangles
to the disjoint union of general odd cycles.
First of all, we shall prove two structural lemmas.

\begin{lem}\label{lemma3.1}
Let $t,\ell,n$ be three positive integers with $n\ge 8t\ell+4\ell+4t-6$.
Let $G$ be a graph on $n$ vertices with $\delta(G)\ge \lfloor\frac{n}{2}\rfloor$,
and $S\subseteq V(G)$ with $|S|\le (t-1)(2\ell+1)$.
If $G-S$ contains a triangle $C^*$, then $G-S$ also contains a $(2\ell+1)$-cycle.
\end{lem}

\begin{proof}
The result holds trivially for $\ell=1$.
Assume now that $\ell\ge 2$.
Set $G'=G-S$ and $C^*=u_0v_0w_0u_0$.
Note that $\delta(G)\ge \lfloor\frac{n}{2}\rfloor$ and $n\ge 8t\ell+4\ell+4t-6$. Then,
$$\delta(G')\ge \delta(G)-|S|\ge \Big\lfloor\frac{n}{2}\Big\rfloor-(t-1)(2\ell+1)> 3\ell.$$
Hence, there exist three vertices $u_1,v_1,w_1$ such that
$u_1\in N_{G'}(u_0)\setminus V(C^*)$, $v_1\in N_{G'}(v_0)\setminus (V(C^*)\cup\{u_1\})$
and $w_1\in N_{G'}(w_0)\setminus (V(C^*)\cup\{u_1,v_1\})$.
Now let $H_0=C^*$. Moreover, we define a subgraph $H_1\subseteq G'$ with $V(H_1)=V(H_0)\cup\{u_1,v_1,w_1\}$ and $E(H_1)=E(H_0)\cup\{u_0u_1,v_0v_1,w_0w_1\}$.
Similarly, there exist three vertices $u_2,v_2,w_2$ such that
$u_2\in N_{G'}(u_1)\setminus V(H_1)$,
$v_2\in N_{G'}(v_1)\setminus (V(H_1)\cup \{u_2\})$ and $w_2\in N_{G'}(w_1)\setminus (V(H_1)\cup \{u_2,v_2\}).$
Repeat the above steps, we can obtain a sequence of subgraphs $H_0,\cdots,H_{\ell-1}$ such that
$V(H_i)=V(H_{i-1})\cup\{u_{i},v_{i},w_{i}\}$ and $$E(H_i)=E(H_{i-1})\cup\{u_{i-1}u_i,v_{i-1}v_i,w_{i-1}w_i\}$$
for $1\le i \le \ell-1$. Note that $|H_i|=3i+3$ for each $i\in \{0,\ldots,\ell-1\}$.
Then we can easily check that $\frac{n+3}4\geq|S|+|H_{\ell-1}|+\frac14$.
Furthermore, for each $x\in \{u_{\ell-1},v_{\ell-1},w_{\ell-1}\}\subseteq V(H_{\ell-1})$
we can see that
\begin{eqnarray*}
|N_{G'}(x)\setminus V(H_{\ell-1})|&\ge& d_{G'}(x)-(|H_{\ell-1}|-1)  \nonumber\\
  &\ge& \delta(G')-|H_{\ell-1}|+1  \nonumber\\
  &\ge& \frac{n-1}{2}-|S|-|H_{\ell-1}|+1  \nonumber\\
  &=& \frac13\Big(n-|S|-|H_{\ell-1}|\Big)+\frac23\Big(\frac{n+3}{4}-|S|-|H_{\ell-1}|\Big).
\end{eqnarray*}
Thus we have $$3|N_{G'}(x)\setminus V(H_{\ell-1})|>n-|S|-|H_{\ell-1}|=|V(G')\setminus V(H_{\ell-1})|.$$
By the pigeonhole principle,
there exists some $y\in V(G')\setminus V(H_{{\ell-1}})$ such that
$y$ is adjacent to at least two vertices, say $v_{\ell-1}$ and $w_{\ell-1}$, of $\{u_{\ell-1},v_{\ell-1},w_{\ell-1}\}$.
Hence, $G'[\{y,v_0,\ldots,v_{\ell-1},w_0,\ldots,w_{\ell-1}\}]$ contains
a cycle of length $2\ell+1$, as $v_0w_0\in E(H_0)$.
The result follows.
\end{proof}

\begin{lem}\label{lemma3.2}
Let $t,k,n$ be three integers with $t\ge 2$, $k\ge \lfloor\frac{19t-9}{2}\rfloor$ and $n\ge \left\lfloor\frac{(k-t)^2}{4\lfloor\frac{t+1}{2}\rfloor}\right\rfloor+(k+1)$.
If $G$ is a graph of order $n$ with $e(G)\geq ex(n,tC_3)$ and $\delta(G)\le \lfloor\frac{n}{2}\rfloor-1$,
then there exists an induced subgraph $G'\subseteq G$ on $n'\geq k$ vertices with
$e(G')\geq ex(n',tC_3)+1$ and $\delta(G')\ge \lfloor\frac{n'}{2}\rfloor$.
\end{lem}

\begin{proof}
By Lemma \ref{lemma-2.1}, for any integer $n^*\ge \lfloor\frac{19t-7}{2}\rfloor$ we have
\begin{align}\label{align1}
ex(n^*,tC_3)-ex(n^*-1,tC_3)=\Big\lfloor\frac{n^*+t-1}{2}\Big\rfloor.
\end{align}
Since $\delta(G)\leq\lfloor\frac{n}{2}\rfloor-1$,
there is a vertex $u_0\in V(G)$ such that
$d_G(u_0)\le \lfloor\frac{n-2}{2}\rfloor$.
Set $G_0=G$ and $G_1=G_0-\{u_0\}$.
Combining $e(G_0)\geq ex(n,tC_3)$, $d_{G_0}(u_0)\le \lfloor\frac{n-2}{2}\rfloor$ and \eqref{align1} gives
\begin{align}\label{align2}
e(G_1)=e(G_0)-d_{G_0}(u_0)\ge ex(n-1,tC_3)+\Big\lfloor\frac{t+1}{2}\Big\rfloor,
\end{align}
as $\lfloor\frac{n+t-1}{2}\rfloor-\lfloor\frac{n-2}{2}\rfloor
\geq\lfloor\frac{t+1}{2}\rfloor$.
Now, if $\delta(G_1)\ge \lfloor\frac{n-1}{2}\rfloor$,
then we define $G'=G_1$ and we are done.
Otherwise, there is a vertex $u_1\in V(G_1)$ such that $d_{G_1}(u_1)\le \lfloor\frac{n-3}{2}\rfloor$.
Then, we set $G_2=G_1-\{u_1\}$. By \eqref{align1} and \eqref{align2}, we obtain
$$e(G_2)=e(G_1)-d_{G_1}(u_1)\ge ex(n-2,tC_3)+2\Big\lfloor\frac{t+1}{2}\Big\rfloor,$$
as $\lfloor\frac{n+t-2}{2}\rfloor-\lfloor\frac{n-3}{2}\rfloor\geq
\lfloor\frac{t+1}{2}\rfloor$.
Repeating the above steps,
we obtain either a $G_i$ for some $i\leq n-k-1$
such that it is a desired induced subgraph or a sequence of induced subgraphs $G_0,G_1,\cdots,G_{n-k}$
such that $|G_i|=n-i$ and
\begin{align}\label{align2'}
e({G_i})\ge ex(n-i,tC_3)+i\Big\lfloor\frac{t+1}{2}\Big\rfloor
\end{align}
for $1\le i\le n-k$.
Since $n\ge \left\lfloor\frac{(k-t)^2}{4\lfloor\frac{t+1}{2}\rfloor}\right\rfloor+(k+1)$,
we have
\begin{align}\label{align2''}
(n-k)\Big\lfloor\frac{t+1}{2}\Big\rfloor> \frac{(k-t)^2}{4}\geq \binom{k-t+1}{2}-\Big\lfloor\frac{(k-t+1)^2}{4}\Big\rfloor.
\end{align}
From Lemma \ref{lemma-2.1} we know that
$$ex(k,tC_3)=\binom{t-1}{2}+(t-1)(k-t+1)+\Big\lfloor\frac{(k-t+1)^2}{4}\Big\rfloor.$$
Combining the above equality with (\ref{align2'}) and (\ref{align2''}), we obtain
\begin{eqnarray*}
e(G_{n-k})\geq ex(k,tC_3)+(n-k)\Big\lfloor\frac{t+1}{2}\Big\rfloor
  >\binom{t-1}{2}+(t-1)(k-t+1)+\binom{k-t+1}{2}=\binom{k}{2},
\end{eqnarray*}
contradicting  $|G_{n-k}|=k$.
Hence, $G_i$ is a desired induced subgraph for some $i\leq n-k-1.$
\end{proof}

Having Lemmas \ref{lemma3.1} and \ref{lemma3.2},
we are now ready to give the proof of Theorem \ref{theorem1.1}.
Recall that $t\ge 2$, $\ell\ge 2$ and $n\ge \left\lfloor\frac{(8t\ell+4\ell+3t-6)^2}{4\lfloor\frac{t}{2}\rfloor}
\right\rfloor+8t\ell+4t+4\ell-5$.
For convenience, we denote $G^*=K_{t-1}+T_{n-t+1,2}$.

\begin{proof}
By Lemma \ref{lemma-2.1}, we have $e(G^*)=ex(n,tC_3)$ for $t\ge 2$ and $n\ge \lfloor\frac{19t-9}{2}\rfloor$.
Moreover, we can easily check that $G^*$ contains at most $t-1$ vertex-disjoint copies of $C_{2\ell+1}$ for each positive integer $\ell$,
as every odd cycle in $G^*$ must occupy at least one vertex in the
$(k-1)$-clique. Let $G$ be an extremal graph with respect to $ex(n,tC_{2\ell+1})$.
Then $$e(G)=ex(n,tC_{2\ell+1})\ge e(G^*)=ex(n,tC_3).$$

Set $k=8t\ell+4\ell+4t-6$. Since $\ell\geq2$, we have $k\geq \big\lfloor\frac{19t-9}2\big\rfloor.$
Suppose now that $\delta(G)\le \lfloor\frac{n}{2}\rfloor-1$. Then
by Lemma~\ref{lemma3.2}, there exists an induced subgraph $G'\subseteq G$ on
$n'\ge k$ vertices such that $e(G')\geq ex(n',tC_3)+1$
and $\delta(G')\ge \lfloor\frac{n'}{2}\rfloor$.
Furthermore, by Lemma~\ref{lemma-2.1}, $G'$ contains $t$ vertex-disjoint triangles $C^1,C^2,\dots,C^t$.

Let $S_1=\cup_{i=2}^{t}V(C^i)$.
Then $|S_1|=3(t-1)\le(t-1)(2\ell+1)$, and
$G'-S_1$ contains a triangle $C^1$.
By Lemma~\ref{lemma3.1}, $G'-S_1$ also contains a $(2\ell+1)$-cycle ${C^1}^*$.
Let $S_2=V({C^1}^*)\cup (\cup_{i=3}^{t}V(C^i))$.
Then $|S_2|=(2\ell+1)+3(t-2)\le (t-1)(2\ell+1)$, and
$G'-S_2$ contains a triangle $C^2$.
Again by Lemma~\ref{lemma3.1}, $G'-S_2$ also contains a $(2\ell+1)$-cycle ${C^2}^*$.

Repeating the above steps, we obtain a sequence of subsets $S_1,\cdots,S_t$ such that
$$S_j=\big(\cup_{i=1}^{j-1}V({C^i}^*)\big)\cup \big(\cup_{i=j+1}^{t}V(C^i)\big)$$
and $G'-S_j$ contains a $(2\ell+1)$-cycle ${C^j}^*$ for $2\le j \le t$.
Hence, $G'$ contains $t$ disjoint $(2\ell+1)$-cycles ${C^1}^*,\dots,{C^t}^*$,
contradicting the fact that $G$ is $tC_{2\ell+1}$-free.
Therefore, $\delta(G)\ge \lfloor\frac{n}{2}\rfloor$.

Recall that $e(G)\ge ex(n,tC_3)$. Then
by Lemma \ref{lemma-2.1}, if $G\not\cong G^*$,
then $G$ contains $t$ vertex-disjoint triangles.
Furthermore, by Lemma \ref{lemma3.1} and a similar way as above to $G'$,
we can find $t$ vertex-disjoint $(2\ell+1)$-cycles in $G$, a contradiction.
Therefore, $G\cong G^*$.
This completes the proof of Theorem \ref{theorem1.1}.
\end{proof}

\section{Proof of Theorem \ref{theorem1.3}}\label{section4}

In this section, we give a proof of Theorem \ref{theorem1.3}.
By Lemma \ref{lemma2.4}, it holds directly for $t=1$.
In the following, assume that $t\ge 2$
and $G$ is an extremal graph with respect to $spex(n,tC_{2\ell+1})$.
Clearly, $G$ is connected
(otherwise, we can add a new edge between two distinct components of $G$).
By Perron-Frobenius theorem, there exists a positive unit eigenvector
$X=(x_1,\ldots,x_n)^T$ corresponding to $\rho(G)$.
Assume that $u^*\in V(G)$ with $x_{u^*}=\max\{x_i~|~i\in V(G)\}$.
We also choose a positive constant $\eta<\frac{1}{75}$,
which will be frequently used
in the proof.
Let $G^*=K_{t-1}+T_{n-t+1,2}$, where $G^*=T_{n,2}$ for $t=1$.
We shall prove $G\cong G^*$ for $n$ sufficiently large.

\begin{lem}\label{lemma4.1}
$\rho(G)\ge\frac{n}{2}+(t-1)-\frac{t^2}{2n}.$
\end{lem}

\begin{proof}
By Theorem \ref{theorem1.1}, $G^*$ is an extremal graph with respect to $ex(n,tC_{2\ell+1})$.
Since $e(T_{n-t+1,2})=\big\lfloor\frac{(n-t+1)^2}4\big\rfloor
\geq\frac{(n-t+1)^2-1}4,$
we have
$$e(G^*)=e(K_{t-1})+e(T_{n-t+1,2})+(t-1)(n-t+1)\ge \frac{1}{4}n^2+\frac{t-1}{2}n-\frac{t^2}{4}. $$
Using the Rayleigh quotient gives
$$\rho(G)\ge \rho(G^*)\ge \frac{\mathbf{1}^TA(G^*)\mathbf{1}}{\mathbf{1}^T\mathbf{1}}=\frac{2e(G^*)}{n}\ge \frac{n}{2}+(t-1)-\frac{t^2}{2n},$$
as desired.
\end{proof}

\begin{lem}\label{lemma4.2}
For $n$ sufficiently large, $e(G)\ge \big(\frac{1}{4}-\frac12\eta^2\big)n^2.$
Furthermore, $G$ admits a partition $V(G)=V_1\cup V_2$ such that $e(V_1,V_2)$ attains the maximum, $e(V_1)+e(V_2)\le \frac12\eta^2 n^2$ and $\big||V_i|-\frac{n}{2}\big|\le\eta n$ for $i\in \{1,2\}.$
\end{lem}

\begin{proof}
Note that $\chi(tC_{2l+1})=3$ and $G$ is $tC_{2l+1}$-free.
Moreover, by Lemma \ref{lemma4.1}, $\rho(G)\ge\frac{n}{2}+(t-1)-\frac{t^2}{2n}$.
Let $\varepsilon$ be a positive constant with $\varepsilon<\frac12\eta^2$.
Then by Lemma \ref{lemma-2.3},
$e(G)\geq \frac{1}{4}n^2-\frac12\eta^2 n^2$, and
there exists a bipartition $V(G)=U_1\cup U_2$ such that
$\lfloor\frac{n}{2}\rfloor\le |U_1|\leq|U_2|\le \lceil\frac{n}{2}\rceil$ and $e(U_1)+e(U_2)\le \frac12\eta^2 n^2$.
We now select a new bipartition $V(G)=V_1\cup V_2$
such that $e(V_1,V_2)$ attains the maximum.
Then $e(V_1)+e(V_2)$ attains the minimum, and $$e(V_1)+e(V_2)\le e(U_1)+e(U_2)\le \frac12\eta^2 n^2.$$
On the other hand, assume that $|V_1|=\frac n2-a$, then $|V_2|=\frac n2+a$.
Thus,
$$e(G)\le |V_1||V_2|+e(V_1)+e(V_2)\le \frac{1}{4}n^2-a^2+\frac12\eta^2 n^2.$$
Combining $e(G)\ge \frac{1}{4}n^2-\frac12\eta^2 n^2$ gives $a^2\le\eta^2 n^2$,
and so $|a|\le\eta n$.
\end{proof}

In the following, we shall define two vertex subsets $U$ and $W$ of $G$.

\begin{lem}\label{lemma4.3}
Let $U=\{v\in V(G)~|~d_G(v)\le \big(\frac{1}{2}-4\eta\big)n\}.$
Then we have $|U|\le \eta n$.
\end{lem}

\begin{proof}
Suppose to the contrary that $|U|>\eta n$,
then there exists $U'\subseteq U$ with $|U'|=\lfloor\eta n\rfloor$.
Moreover, by Lemma \ref{lemma4.2},
we have $e(G)\ge \big(\frac{1}{4}-\frac{1}{2}\eta^2\big)n^2$.
Now set $n'=|G-U'|=n-\lfloor\eta n\rfloor$.
Then $n'-1<\big(1-\eta)n$. Thus,
\begin{eqnarray*}
e(G-U')&\ge& e(G)-\sum_{v\in U'}d_G(v)\nonumber\\
  &\ge&  \Big(\frac{1}{4}-\frac{\eta^2}{2}\Big)n^2-\eta n\Big(\frac{1}{2}-4\eta\Big)n\nonumber\\
  &=& \frac14\big(1-2\eta+14\eta^2\big)n^2\nonumber\\
  &>& \frac14\big(n'-1+t\big)^2
\end{eqnarray*}
for sufficiently large $n$.
We can further check that $\frac14(n'+t-1)^2>e(K_{t-1}+T_{n'-t+1,2}).$
Hence, $e(G-U')>e(K_{t-1}+T_{n'-t+1,2})$. By Theorem \ref{theorem1.1},
$G-U'$ contains $t$ vertex-disjoint $(2\ell+1)$-cycles,
contradicting the fact that $G$ is $tC_{2\ell+1}$-free.
\end{proof}

\begin{lem}\label{lemma4.4}
Let $W=W_1\cup W_2$, where $W_i=\{v\in V_i~|~d_{V_i}(v)\ge 2\eta n\}$
and $d_{V_i}(v)=|N_G(v)\cap V_i|$ for $i\in\{1,2\}$.
Then we have $|W|\le\frac{1}{2}\eta n$.
\end{lem}

\begin{proof}
For $i\in \{1,2\}$,
$$2e(V_i)=\sum_{v\in V_i}d_{V_i}(v)\ge
\sum_{v\in W_i}d_{V_i}(v)\ge |W_i|\cdot 2\eta n.$$
Combining Lemma \ref{lemma4.2}, we have
$$\frac12\eta^2 n^2\ge e(V_1)+e(V_2)\ge \big(|W_1|+|W_2|\big)\eta n=|W|\eta n.$$
Therefore, $|W|\le \frac12\eta n$.
\end{proof}

In the following three lemmas,
we focus on constructing $(2\ell+1)$-cycles
in distinct induced subgraphs of the spectral extremal graph $G$.

\begin{lem}\label{lemma4.5}
For arbitrary $R\subseteq V(G)$ with $|R|\le t(2\ell+1)$,
if there exists an edge within $V_i\setminus(U\cup W\cup R)$
for some $i\in \{1,2\}$, then
$G-(U\cup W\cup R)$ contains a $(2\ell+1)$-cycle.
\end{lem}

\begin{proof}
Let $V'=V_1'\cup V_2'$, where $V_i'=V_i\setminus (U\cup W\cup R)$
for $i\in \{1,2\}$.
Moreover, we may assume that $\widehat{i}\in\{1,2\}\setminus \{i\}$.
We first claim that for each vertex $u\in V_i'$,
\begin{align}\label{align3}
 |N_{V'}(u)|\geq |N_{V_{\widehat{i}}'}(u)|>\frac{2}{5}n,
\end{align}
where $N_{V'}(u)=N_G(u)\cap V'$.
Since $u\notin U\cup W$, we know that $d_{V_{i}}(u)<2\eta n$
and $d_G(u)>(\frac{1}{2}-4\eta)n$.
Recall that $V_1\cup V_2$ is a bipartition of $V(G)$.
Thus $d_{V_{\widehat{i}}}(u)=d_G(u)-d_{V_i}(u)>(\frac{1}{2}-6\eta)n$.
Combining Lemmas \ref{lemma4.3} and \ref{lemma4.4} gives
$$|N_{V_{\widehat{i}}'}(u)|
\geq|N_{V_{\widehat{i}}}(u)|-\big(|U|+|W|+|R|\big)
>\Big(\frac{1}{2}-6\eta\Big)n-\frac32\eta n-t(2\ell+1)>
\frac{2}{5}n,$$
as the constant $\eta<\frac1{75}$ and $n$ is sufficiently large.
Thus, \eqref{align3} follows.

Now let $u_0v_0$ be an arbitrary edge within $V_i'$.
From \eqref{align3} we know that both $|N_{V_{\widehat{i}}'}(u_0)|>\frac{2}{5}n$ and $|N_{V_{\widehat{i}}'}(v_0)|>\frac{2}{5}n$.
Moreover, by Lemma \ref{lemma4.2},
$|V_{\widehat{i}}'|\leq|V_{\widehat{i}}|\leq\frac n2+\eta n$.
Thus,
$$\big|N_{V_{\widehat{i}}'}(u_{0})\cap N_{V_{\widehat{i}}'}(v_{0})\big|\geq \big|N_{V_{\widehat{i}}'}(u_{0})\big|+\big|N_{V_{\widehat{i}}'}(v_{0})\big|
-\big|V_{\widehat{i}}'\big|>\frac{3}{10}n-\eta n>0,$$
and hence there exists a vertex $w_0\in N_{V_{\widehat{i}}'}(u_{0})\cap N_{V_{\widehat{i}}'}(v_{0})$.
Since $w_0\in V_{\widehat{i}}'$,
it follows from \eqref{align3} that $|N_{V'}(w_{0})|\geq|N_{V_{i}'}(u)|>\frac25n$.

Let $H_0=G[\{u_0,v_0,w_0\}]$.
Then $H_0\cong C_3$ and $H_0\subseteq G-(U\cup W\cup R)$.
If $\ell=1$, then $H_0$ is a desired $(2\ell+1)$-cycle.
Assume now that $\ell\ge 2$.
Since $|N_{V'}(u)|>\frac25n$ for each $u\in V(H_0)$,
there exist $u_1,v_1,w_1\in V'$ such that
$u_1\in N_{V'}(u_0)\setminus V(H_0)$, $v_1\in N_{V'}(v_0)\setminus (V(H_0)\cup \{u_1\})$
and $w_1\in N_{V'}(w_0)\setminus (V(H_0)\cup \{u_1,v_1\})$.
Then, we define a subgraph $H_1\subseteq G$ with $V(H_1)=V(H_0)\cup\{u_1,v_1,w_1\}$
and $E(H_1)=E(H_0)\cup\{u_0u_1,v_0v_1,w_0w_1\}$.
Similarly, there exist $u_2,v_2,w_2$ such that
$u_2\in N_{V'}(u_1)\setminus V(H_1)$,
$v_2\in N_{V'}(v_1)\setminus (V(H_1)\cup \{u_2\})$ and $w_2\in N_{V'}(w_1)\setminus (V(H_1)\cup \{u_2,v_2\}).$
Repeating the above steps,
we obtain a sequence of subgraphs $H_0,H_1,\cdots,H_{\ell-1}$ such that
$V(H_j)=V(H_{j-1})\cup\{u_{j},v_{j},w_{j}\}$ and $E(H_j)=E(H_{j-1})\cup\{u_{j-1}u_j,v_{j-1}v_j,w_{j-1}w_j\}$
for $1\le j\le \ell-1$.
Then, $|H_{\ell-1}|=3\ell$
and $H_{\ell-1}\subseteq G-(U\cup W\cup R)$.
Set $V''=V'\setminus V(H_{\ell-1})$.
For each $u\in \{u_{\ell-1},v_{\ell-1},w_{\ell-1}\}$,
we have
$|N_{V''}(u)|\geq|N_{V'}(u)|-|H_{\ell-1}|+1>\frac{2}{5}n-3\ell+1,$
and thus $$|N_{V''}(u_{\ell-1})|+|N_{V''}(v_{\ell-1})|
            +|N_{V''}(w_{\ell-1})|>n> |V''|$$
for $n$ sufficiently large.
This implies that
there exists $w\in V''$ such that
$w$ is adjacent to at least two vertices, say $u_{\ell-1}$ and $v_{\ell-1}$, of $\{u_{\ell-1},v_{\ell-1},w_{\ell-1}\}$.
Therefore, $G-(U\cup W\cup R)$ contains a $(2\ell+1)$-cycle
$u_0\ldots u_{\ell-1}wv_{\ell-1}\ldots v_0u_0$.
The proof is completed.
\end{proof}

\begin{lem}\label{lemma4.6}
For arbitrary $R\subseteq V(G)$ with $|R|\le t(2\ell+1)$,
if there exists a vertex $u_0\in W\setminus U$, then
$G-\big((U\cup W\cup R)\setminus\{u_0\}\big)$ contains a $(2\ell+1)$-cycle.
\end{lem}

\begin{proof}
Since $V(G)=V_1\cup V_2$,
we may assume without loss of generality that $u_0\in V_1$.
Then by the definitions of $U$ and $W$, we have $$d_G(u_0)>\Big(\frac{1}{2}-4\eta\Big)n
~~~ \mbox{and} ~~~ d_{V_1}(u_0)\ge 2\eta n.$$
Moreover, by Lemmas \ref{lemma4.3} and \ref{lemma4.4},
$|U|\le\eta n$ and $|W|\le \frac12 \eta n$. Thus,
$$\big|N_{V_1\setminus(U\cup W\cup R)}(u_0)\big|\ge d_{V_1}(u_0)-(|U|+|W|+|R|)\ge\frac12\eta n-t(2\ell+1)>0.$$
Then, there exists a vertex $v_0$ in $N_{V_1}(u_0)\setminus(U\cup W\cup R)$.
Again by the definitions of $U$ and $W$, we can see that
$d_G(v_0)>(\frac{1}{2}-4\eta)n$ and $d_{V_{1}}(v_0)<2\eta n$.
It follows that
\begin{align}\label{align4}
d_{V_{2}}(v_0)=d_G(v_0)-d_{V_1}(v_0)>\Big(\frac{1}{2}-6\eta\Big)n.
\end{align}
Recall that $V(G)=V_1\cup V_2$ is a bipartition of $V(G)$ such that $e(V_1,V_2)$ attains the maximum. Hence,
$d_{V_1}(u_0)\le \frac{1}{2}d_G(u_0)$.
Since $d_G(u_0)>(\frac{1}{2}-4\eta)n$, we get that
\begin{align}\label{align5}
 d_{V_2}(u_0)=d_G(u_0)-d_{V_1}(u_0)\ge \frac{1}{2}d_G(u_0)>
 \Big(\frac{1}{4}-2\eta\Big)n.
\end{align}
Furthermore, Lemma \ref{lemma4.2} gives $|V_2|\leq\frac n2+\eta n.$
Combining with \eqref{align4} and \eqref{align5}, we obtain
$$\left|N_{V_2}(u_0)\cap N_{V_2}(v_0)\right|\ge|N_{V_2}(u_0)|+|N_{V_2}(v_0)|-|V_2|\ge \Big(\frac{1}{4}-9\eta\Big)n.$$
Note that $\eta<\frac1{75}$ and $n$ is sufficiently large.
It follows that
$$\Big|\big(N_{V_2}(u_0)\cap N_{V_2}(v_0)\big)\setminus (U\cup W\cup R)\Big|
\ge\Big(\frac{1}{4}-9\eta\Big)n-\frac{3}{2}\eta n-t(2\ell+1)>0.$$
Hence, there exists
$w_0\in \big(N_{V_2}(u_0)\cap N_{V_2}(v_0)\big)\setminus (U\cup W\cup R)$.
Let $H_0=G[\{u_0,v_0,w_0\}]$.
Then $H_0\cong C_3$ and $H_0\subseteq G-\big((U\cup W\cup R)\setminus\{u_0\}\big)$.
For $\ell=1$, $H_0$ is a $(2\ell+1)$-cycle.
For $\ell\ge 2$,
using the same method as in the proof of Lemma \ref{lemma4.5},
we can find a $(2\ell+1)$-cycle in $G-\big((U\cup W\cup R)\setminus\{u_0\}\big)$.
\end{proof}

\begin{lem}\label{lemma4.7}
Let $\nu=\sum_{i=1}^2\nu\big(G[V_i\setminus (U\cup W)]\big)$.
Then $\nu\leq t-1$.
Moreover, $G-(U\cup W)$ contains at least $\nu$ disjoint $(2\ell+1)$-cycles.
\end{lem}

\begin{proof}
The case $\nu=0$ is trivial. Now assume that $\nu\ge1$, and let
$u_1u_2,\ldots,u_{2\nu-1}u_{2\nu}$ be $\nu$ independent edges in $G[V_1\setminus(U\cup W)]\cup G[V_2\setminus(U\cup W)]$.
Then, we set $R_0=\{u_j~|~j=1,2,\ldots,2\lambda\}$
and $R_1=R_0\setminus \{u_1,u_2\}$, where $\lambda=\min\{\nu,t\}$.
Since $u_1u_2$ is an edge within $V_i\setminus (U\cup W\cup R_1)$ for some $i\in\{1,2\}$,
Lemma \ref{lemma4.5} indicates that
$G-(U\cup W\cup R_1)$ contains a $(2\ell+1)$-cycle $C^1$.
Let $R_2=\big(R_1\setminus \{u_3,u_4\}\big)\cup V(C^1)$.
Again by Lemma \ref{lemma4.5}, $G-(U\cup W\cup R_2)$ contains a $(2\ell+1)$-cycle $C^2$,
as $u_3u_4$ is an edge within $V_i\setminus (U\cup W\cup R_2)$ for some $i\in\{1,2\}$.

Repeating the above steps, we obtain a sequence of vertex subsets
$R_1,\cdots,R_\lambda$ such that
$R_j=\big(R_{j-1}\setminus\{u_{2j-1},u_{2j}\}\big)\cup\big(\cup_{k=1}^{j-1}V(C^k)\big)$
and $G-(U\cup W\cup R_j)$ contains a $(2\ell+1)$-cycle $C^j$
for each $j\in \{2,\dots,\lambda\}$.
Clearly, $|R_j|\leq(\lambda-1)(2\ell+1)$ for $1\leq j\leq\lambda$;
moreover, $C^1,C^2,\dots,C^\lambda$ are disjoint cycles in $G-(U\cup W)$.
Since $G$ is $tC_{2\ell+1}$-free, we have $\lambda\le t-1$.
Combining $\lambda=\min\{\nu,t\}$ gives $\nu=\lambda\leq t-1$, and thus $C^1,C^2,\dots,C^\nu$ are
disjoint $(2\ell+1)$-cycles in $G-(U\cup W)$.
\end{proof}

In the following two lemmas,
we shall give two local structural properties of $G$.

\begin{lem}\label{lemma4.8}
For $i\in \{1,2\}$,
we have $\Delta\big(G[V_i\setminus(U\cup W)]\big)<t(2\ell+1)$.
\end{lem}

\begin{proof}
Our proof is by contradiction.
Without loss of generality,
suppose that there exists a vertex $u_0\in V_1\setminus (U\cup W)$
such that $d_{V_1\setminus (U\cup W)}(u_0)\ge t(2\ell+1)$.
Since $u_0\notin W$, we get $d_{V_1}(u_0)<2\eta n$ by the definition of $W$.
On the other hand, by Lemma \ref{lemma4.2},
$|V_1|\geq\frac n2-\eta n$,
and so
$$|V_1\setminus (U\cup W)|\ge |V_1|-|U|-|W|
\ge \Big(\frac{1}{2}-\frac52\eta\Big)n.$$
Hence, $|V_1\setminus(U\cup W)|>d_{V_1}(u_0)$, as $\eta<\frac1{75}.$
This implies that there exist vertices in $V_1\setminus(U\cup W)$
which are not adjacent to $u_0$.
Let $G'$ be the graph obtained from $G$ by adding all possible edges from $u_0$ to
$V_1\setminus(U\cup W)$.
Then $\rho(G')>\rho(G)$.
Since $G$ is extremal with respect to $spex(n,tC_{2\ell+1})$,
$G'$ must contain a subgraph $H$ isomorphic to $tC_{2\ell+1}$.
From the construction of $G'$,
we can further see that $u_0\in V(C)$ for some $(2\ell+1)$-cycle $C$ in $H$.
Set $H'=H-V(C).$ Then $H'\subseteq G$.
Since $d_{V_1\setminus (U\cup W)}(u_0)\ge t(2\ell+1)$ while $|H'|=(t-1)(2\ell+1)$,
there exists a vertex $v_0$ with $v_0\in N_{V_1\setminus (U\cup W)}(u_0)$
and $v_0\notin V(H')$ in $G$.

Now setting $R=V(H')$ in Lemma \ref{lemma4.5},
and noticing that $u_0v_0$ is an edge within $V_1\setminus (U\cup W\cup R)$,
we obtain that $G-(U\cup W\cup R)$ contains a $(2\ell+1)$-cycle $C'$.
Clearly, $V(C')\cap V(H')=\varnothing$.
Therefore, $C'\cup H'$ is a copy of $tC_{2\ell+1}$ in $G$,
which contradicts the fact that $G$ is $tC_{2\ell+1}$-free.
\end{proof}

\begin{lem}\label{lemma4.9}
For $i\in \{1,2\}$, $G[V_i\setminus (U\cup W)]$ contains an independent set $I_i$ with $|I_i|>|V_i\setminus (U\cup W)|-2(t-1)t(2\ell+1)$.
\end{lem}

\begin{proof}
Assume that $\nu_i=\nu\big(G[V_i\setminus (U\cup W)]\big)$ for $i\in \{1,2\}$.
If $\nu_i=0$, then $V_i\setminus (U\cup W)$ is a desired independent set.
Now assume that $\nu_i\ge 1$, and let $u_1u_2,\ldots,u_{2\nu_i-1}u_{2\nu_i}$
be $\nu_i$ independent edges in $G[V_i\setminus (U\cup W)]$.
Let $$I_i=\big(V_{i}\setminus (U\cup W)\big)\setminus\big(\cup_{j=1}^{2\nu_i}N_{V_{i}\setminus (U\cup W)}(u_j)\big).$$
Then, every vertex in $I_i$ is not adjacent to any vertex in $\{u_1,u_2,\dots,u_{2\nu_i}\}$.
Now, if $G[I_i]$ contains an edge,
then $\nu\big(G[V_i\setminus (U\cup W)]\big)\geq \nu_i+1$,
a contradiction.
Therefore, $I_i$ is an independent set.

From Lemma \ref{lemma4.8} we know that
$\Delta\big(G[V_i\setminus (U\cup W)]\big)<t(2\ell+1)$.
Moreover, $\nu_i\leq\nu\le t-1$ by Lemma \ref{lemma4.7}.
Thus, we can see that
$$|V_i\setminus (U\cup W)|-|I_i|=|\cup_{j=1}^{2\nu_i}N_{V_i\setminus (U\cup W)}(u_j)|\leq2\nu_i\Delta\big(G[V_i\setminus (U\cup W)]\big)
<2(t-1)t(2\ell+1).$$
The result follows.
\end{proof}

In the following three lemmas,
we will give exact characterizations of $U$ and $W$.
Since $|W|\leq\frac12\eta n<n$ by Lemma \ref{lemma4.4},
we may choose a vertex $v^*$ such that
$x_{v^*}=\max\{x_v~|~v\in V(G)\setminus W\}$.
We will see that $v^*\notin U$.
Recall that $x_{u^*}=\max\{x_v~|~v\in V(G)\}$. Then
$\rho(G)x_{u^*}\le |W|x_{u^*}+(n-|W|)x_{v^*}.$
Moreover,
$\rho(G)>\frac n2$ by Lemma \ref{lemma4.1}.
It follows that
\begin{align}\label{align6}
x_{v^*}\ge \frac{\rho(G)-|W|}{n-|W|}x_{u^*}\ge \frac{\rho(G)-|W|}{n}x_{u^*}
    >\frac{1}{2}\big(1-\eta\big)x_{u^*}.
\end{align}
Since $\eta<\frac1{75}$, we have $x_{v^*}>\frac{2}{5}x_{u^*}$.
On the other hand,
$$\rho(G)x_{v^*}
=\sum_{v\in N_W(v^*)}x_v+\sum_{v\in N_{G-W}(v^*)}x_v
\le |W|x_{u^*}+d_G(v^*)x_{v^*}.$$
Combining with $x_{v^*}>\frac{2}{5}x_{u^*}$,
$\rho(G)>\frac n2$ and $|W|\leq \frac12\eta n$, we obtain
$$d_G(v^*)\ge \rho(G)-\frac{x_{u^*}}{x_{v^*}}|W|\geq\rho(G)-\frac{5}{2}|W|> \Big(\frac{1}{2}-\frac{5}{4}\eta\Big)n.$$
Recall that $U=\{v\in V(G)~|~d_G(v)\leq \big(\frac12-4\eta\big)n\}$.
Then $v^*\notin U$, and so $v^*\in V(G)\setminus (U\cup W)$.

Assume now that $v^*\in V_{i^*}\setminus (U\cup W)$ for some $i^*\in\{1,2\}$, and set
$\widehat{i^*}\in\{1,2\}\setminus \{i^*\}$.
Then by Lemma \ref{lemma4.8}, $|N_{V_{i^*}}(v^*)\setminus(U\cup W)|<t(2\ell+1)$.
Thus,
\begin{eqnarray*}
\rho(G)x_{v^*}&=& \sum_{v\in N_{U\cup W}(v^*)}x_v+
                    \sum_{v\in N_{V_{i^*}}(v^*)\setminus(U\cup W)}x_v+
                    \sum_{v\in N_{V_{\widehat{i^*}}}(v^*)\setminus(U\cup W)}x_v \nonumber\\
                    &<& \big(|W|x_{u^*}+|U|x_{v^*}\big)+t(2\ell+1)x_{v^*}+\sum_{v\in V_{\widehat{i^*}}\setminus(U\cup W\cup I_{\widehat{i^*}})}x_v+\sum_{v\in I_{\widehat{i^*}}}x_v\nonumber\\
&\le& \big(|W|x_{u^*}+|U|x_{v^*}\big)+(2t-1)t(2\ell+1)x_{v^*}+
\sum_{v\in I_{\widehat{i^*}}}x_v,
\end{eqnarray*}
where $I_{\widehat{i^*}}$ is an independent set of
$G[V_{\widehat{i^*}}\setminus(U\cup W)]$
such that $\big|V_{\widehat{i^*}}\setminus (U\cup W\cup I_{\widehat{i^*}})\big|<2(t-1)t(2\ell+1)$ (see Lemma \ref{lemma4.9}).
Subsequently,
\begin{align}\label{align7}
\sum_{v\in I_{\widehat{i^*}}}x_v
>\big(\rho(G)-|U|-(2t-1)t(2\ell+1)\big)x_{v^*}-|W|x_{u^*}.
\end{align}

\begin{lem}\label{lemma4.10}
We have $U=\varnothing$.
\end{lem}

\begin{proof}
Suppose to the contrary that there exists $u_0\in U$.
Let $G'$ be the graph obtained from $G$ by deleting edges incident to $u_0$
and joining all possible edges from $I_{\widehat{i^*}}$ to $u_0$.

We claim that $G'$ is $tC_{2\ell+1}$-free.
Otherwise, $G'$ contains a subgraph $H$ isomorphic to $tC_{2\ell+1}$.
From the construction of $G'$,
we can see that $H$ must contain a $(2\ell+1)$-cycle $C'$ with $u_0\in V(C')$.
Set $H'=H-V(C')$. Then $H'\subseteq G$.
Assume that $N_{C'}(u_0)=\{u_1,u_2\}$,
then $u_1,u_2\in I_{\widehat{i^*}}$ by the definition of $G'$.
Since $I_{\widehat{i^*}}\subseteq V_{\widehat{i^*}}\setminus(U\cup W)$,
we have $u_1,u_2\notin U\cup W$.
By the definitions of $U$ and $W$,
we know that
$d_G(u_j)>\big(\frac12-4\eta\big)n$ and $d_{V_{\widehat{i^*}}}(u_j)<2\eta n$
for $j\in\{1,2\}$.
Hence,
$|N_{V_{i^*}}(u_1)|=d_G(u_1)-d_{V_{\widehat{i^*}}}(u_1)>\big(\frac12-6\eta\big)n$.
Similarly, $|N_{V_{i^*}}(u_2)|>\big(\frac12-6\eta\big)n$.
Moreover, $|V_{i^*}|\leq\frac n2+\eta n$ by Lemma \ref{lemma4.2}.
It follows that
$$|N_{V_{i^*}}(u_1)\cap N_{V_{i^*}}(u_2)|\ge|N_{V_{i^*}}(u_1)|+|N_{V_{i^*}}(u_2)|-|V_{i^*}|
>\Big(\frac{1}{2}-13\eta\Big)n.$$
Now, note that $|H|=t(2\ell+1).$
Then $|N_{V_{i^*}}(u_1)\cap N_{V_{i^*}}(u_2)|>|H|$,
and hence we can find a vertex $u\in
\big(N_{V_{i^*}}(u_1)\cap N_{V_{i^*}}(u_2)\big)\setminus V(H)$.
This implies that $G-V(H')$ contains a $(2\ell+1)$-cycle $C''$, which is obtained from
$C'$ by replacing $\{u_0u_1,u_0u_2\}$ with $\{uu_1,uu_2\}$.
Hence, $C''\cup H'$ is a copy of $tC_{2\ell+1}$ in $G$, a contradiction.
Therefore, the above claim holds.

Now, $d_G(u_0)\le (\frac{1}{2}-4\eta)n$ by the definition of $U$.
Recall that $\rho(G)>\frac{n}{2}$ and $|U|\le \eta n$.
Then
\begin{align}\label{align8}
\rho(G)-d_G(u_0)-|U|>3\eta n.
\end{align}
Moreover,
\begin{align}\label{align9}
\sum_{v\in N_G(u_0)}x_{v}=\sum_{v\in N_W(u_0)}x_v+\sum_{v\in N_{G-W}(u_0)}x_v
\le |W|x_{u^*}+d_G(u_0)x_{v^*}.
\end{align}
Recall that $x_{v^*}>\frac25x_{u^*}$ and $|W|\le \frac12\eta n$.
Combining \eqref{align7}, \eqref{align8} and \eqref{align9},
we get that
\begin{eqnarray*}
   \sum_{v\in I_{\widehat{i^*}}}x_{v}\!\!-\!\!\sum_{v\in N_G(u_0)}x_{v}
    &\ge& \sum_{v\in I_{\widehat{i^*}}}x_v-\big(|W|x_{u^*}+d_G(u_0)x_{v^*}\big) \nonumber\\
    &>& \big(\rho(G)\!-\!d_G(u_0)\!-\!|U|\!-\!(2t-1)t(2\ell+1)\big)x_{v^*}\!-\!2|W|x_{u^*}  \nonumber\\
  &>&\big(3\eta n-(2t-1)t(2\ell+1)\big)\frac{2}{5}x_{u^*}-\eta nx_{u^*} \nonumber\\
  &>& \frac{1}{10}\eta nx_{u^*}
\end{eqnarray*}
for $n$ sufficiently large. Thus,
$$\rho(G')-\rho(G) \ge  X^T\big(A(G')-A(G)\big)X
                  = 2x_{u_0}\Big(\sum_{v\in I_{\widehat{i^*}}}x_v-\sum_{v\in N_G(u_0)}x_v\Big)>0,$$
contradicting the fact that $G$ is extremal with respect to $spex(n,tC_{2\ell+1})$.
Hence, $U=\varnothing$.
\end{proof}

\begin{lem}\label{lemma4.11}
 For each $v\in V(G)$, we have $x_v>\frac{2}{5}x_{u^*}$.
\end{lem}

\begin{proof}
Recall that $\rho(G)>\frac{n}{2}$ and
$|W|\leq\frac12 \eta n$.
Then $|W|<\eta\rho(G)$.
Moreover, $U=\varnothing$ by Lemma \ref{lemma4.10}.
Combining (\ref{align7}), we obtain that
$$\sum_{v\in I_{\widehat{i^*}}}x_v
>\big(\rho(G)-(2t-1)t(2\ell+1)\big)x_{v_0}-\eta \rho(G)x_{u^*}. $$
From \eqref{align6} we know that $x_{v^*}>\frac{1}{2}\big(1-\eta\big)x_{u^*}.$
Thus, for $n$ sufficiently large,
$$\sum_{v\in I_{\widehat{i^*}}}x_v>
\Big(\frac{1}{2}-2\eta\Big)\rho(G)x_{u^*}.$$

Now, suppose to the contrary that there exists $u_0\in V(G)$ such that $x_{u_0}\leq\frac{2}{5}x_{u^*}$.
Let $G'$ be the graph obtained from $G$ by deleting edges incident to $u_0$
and joining all edges from $I_{\widehat{i^*}}$ to $u_0$.
By a similar discussion as in the proof of Lemma \ref{lemma4.10},
we claim that $G'$ is $tC_{2\ell+1}$-free.
However,
 $$\sum_{v\in I_{\widehat{i^*}}}x_{v}\!\!-\!\!\sum_{v\in N_G(u_0)}x_{v}=\sum_{v\in I_{\widehat{i^*}}}x_v -\rho(G)x_{u_0}>\left(\frac{1}{2}-2\eta-\frac25\right)\rho(G)x_{u^*}>0,$$
which implies that
$$\rho(G')-\rho(G) \ge  X^T\big(A(G')-A(G)\big)X
                  = 2x_{u_0}\Big(\sum_{v\in I_{\widehat{i^*}}}x_v-\sum_{v\in N_G(u_0)}x_v\Big)>0,$$
contradicting the fact that $G$ is extremal with respect to $spex(n,tC_{2\ell+1})$.
\end{proof}

\begin{lem}\label{lemma4.12}
$|W|=t-1$ and $\nu=0$.
\end{lem}

\begin{proof}
Note that $U=\varnothing$.
By Lemma \ref{lemma4.7}, $\nu=\nu(\cup_{i=1}^2 G[V_i\setminus W])\le t-1$;
and if $\nu\ge 1$, then $G-W$ contains $\nu$ disjoint $(2\ell+1)$-cycles $C^1,C^2,\dots,C^\nu$.

We first claim that $|W|\le t-1-\nu$.
Otherwise, $|W|\ge t-\nu$.
Let $R_0=\{u_1,u_2,\dots,u_{t-\nu}\}$ be a subset of $W$.
Furthermore, we define $R_1=R_0\setminus\{u_1\}$ if $\nu=0$; and $R_1=(R_0\setminus\{u_1\})\cup(\cup_{i=1}^\nu V(C^i))$ if $\nu\ge 1$.
Then $|R_1|\leq (t-1)(2\ell+1).$ By Lemma \ref{lemma4.6},
$G-\big((W\cup R_1)\setminus \{u_1\}\big)$ contains a $(2\ell+1)$-cycle $C^{\nu+1}$,
where $V(C^{\nu+1})\cap R_0\subseteq \{u_1\}$.
If $t-\nu\geq2$, then we further define $R_2=\big(R_1\setminus \{u_2\}\big)\cup V(C^{\nu+1})$.
Clearly, $|R_2|\leq (t-1)(2\ell+1).$
Again by Lemma \ref{lemma4.6},
$G-\big((W\cup R_2)\setminus \{u_2\}\big)$ contains a $(2\ell+1)$-cycle $C^{\nu+2}$,
where $V(C^{\nu+2})\cap R_0\subseteq \{u_2\}$.
Repeating the above steps, we obtain a sequence of vertex subsets
$R_1,R_2,\ldots,R_{t-\nu}$ with
$R_j=\big(R_{j-1}\setminus \{u_{j}\}\big)\cup \big(\cup_{k=1}^{j-1}V(C^{\nu+k})\big)$
and $|R_j|\leq (t-1)(2\ell+1)$ such that
$G-\big((W\cup R_j)\setminus \{u_j\}\big)$ contains a $(2\ell+1)$-cycle $C^{\nu+j}$ for each $j\in \{1,\dots,t-\nu\}$.
Furthermore, $V(C^{\nu+j})\cap R_0\subseteq \{u_j\}$ for $1\leq j\leq t-\nu$.
Thus we can observe that $C^1,C^2,\dots,C^t$ are disjoint,
which contradicts the fact that $G$ is $tC_{2\ell+1}$-free.

Now define $H=\cup_{i=1}^{2}G[V_i\setminus W]$.
Then $\nu(H)=\nu$.
We further claim that
\begin{align}\label{align10}
e(H)\le (t-1)(2t\ell+t+1).
\end{align}
The case $\nu=0$ is trivial.
Assume that $\nu\ge 1$.
By Lemma \ref{lemma4.8},
$\Delta(H)<t(2\ell+1)$.
Recall that $f(\nu,\Delta)=\max\{e(G)~|~\nu(G)\le \nu, \Delta(G)\le \Delta\}$,
and by Lemma \ref{lemma-2.2} $f(\nu,\Delta)\leq \nu(\Delta+1)$.
Thus,
$$e(H)\le f\big(\nu(H),\Delta(H)\big)
\leq f(\nu,t(2\ell+1))\leq\nu\cdot(2t\ell+t+1).
$$
Note that $\nu\leq t-1$. Therefore, \eqref{align10} holds.

Note that $|W|\le t-1-\nu\le t-1$.
It suffices to prove $|W|=t-1$, as it implies that $\nu=0$.
Suppose to the contrary that $|W|\le t-2$.
Take $S\subseteq V_1\setminus W$ with $|S|=t-1-|W|$,
and let $G'$ be the graph obtained from $G$ by deleting all edges in $E(H)$ and
adding all possible edges from $S$ to $V_1\setminus (W\cup S)$.
Clearly, $G'$ is a spanning subgraph of
$K_{|W\cup S|}+K_{|V_1\setminus (W \cup S)|,|V_2\setminus W|}$.
Since $|W\cup S|=t-1$, $G'$ contains at most $t-1$ disjoint odd cycles,
and so $G'$ is $tC_{2\ell+1}$-free.

Recall that $|V_1|\ge \frac12 n-\eta n$, and by Lemma \ref{lemma4.11},
$x_v>\frac25x_{u^*}$ for each $v\in V(G)$. Combining \eqref{align10}, we have
\begin{eqnarray*}
  \rho(G')-\rho(G)  &\ge&   X^T\big(A(G')-A(G)\big)X
    \geq\sum_{u\in S,v\in V_1\setminus (W \cup S)}2x_ux_v-\sum_{uv\in E(H)}2x_ux_v \nonumber\\
  &\ge& |S|\big(\frac{n}{2}\!-\!\eta n\!-\!t\!+1\big)\frac{8}{25}x_{u^*}^2 \!-\!(t-1)(2t\ell\!+\!t\!+\!1)2x_{u^*}^2\nonumber\\
  &>&0,
\end{eqnarray*}
contradicting the fact that $G$ is an extremal graph with respect to $spex(n,tC_{2\ell+1})$.
\end{proof}

In the following, we complete the proof of Theorem \ref{theorem1.3}.

\begin{proof}
Recall that $G^*=K_{t-1}+T_{n-t+1,2}$ and we shall prove $G\cong G^*$.
We first look for a $(t-1)$-clique in which each vertex dominates all vertices of $G$.
By Lemma \ref{lemma4.12}, we know that $|W|=t-1$.
It suffices to show that $d_G(u)=n-1$ for each $u\in W$.

Suppose to the contrary that there exists a vertex $u\in W$ with $d(u)<n-1$.
Then we can select a non-neighbor $v$ of $u$ in $G$.
Let $G'=G+\{uv\}$. Then $\rho(G')>\rho(G)$.
Since $G$ is extremal with respect to $spex(n,tC_{2\ell+1})$,
$G'$ contains a subgraph $H$ isomorphic to $tC_{2\ell+1}$,
where $uv\in E(H)$.
More precisely, $H$ contains a $(2\ell+1)$-cycle $C$ with $uv\in V(C)$.
Set $H'=H-V(C)$. Then $H'\subseteq G$,
and by Lemma \ref{lemma4.6},
$G-\big((W\cup V(H'))\setminus\{u\}\big)$ contains a $(2\ell+1)$-cycle $C'$.
Since $u\notin V(H')$, $H'\cup C'$ is a copy of $tC_{2\ell+1}$ in $G$,
a contradiction.
Therefore, $d_G(u)=n-1$ for each $u\in W$.

Let $|V_i\setminus W|=n_i$ for $i\in \{1,2\}$.
Assume without loss of generality that $n_1\ge n_2$.
By Lemma \ref{lemma4.12}, $\nu=\nu(\cup_{i=1}^2 G[V_i\setminus W])=0$,
and thus $G-W\subseteq K_{n_1,n_2}$.
Since $G$ is extremal, we have $G-W\cong K_{n_1,n_2}$.
To show $G\cong G^*$,
it suffices to show $G-W\cong T_{n-t+1,2}$, or equivalently, $n_1-n_2\le 1$.

Suppose to the contrary that $n_1\geq n_2+2$.
By symmetry, we may assume $x_u=x_i$ for each $u\in V_i\setminus W$
and $i\in\{1,2\}$.
Moreover, let $x_u=x_3$ for each $u\in W$. Thus,
\begin{align*}
\rho(G)x_1=n_2x_2+(t-1)x_3,
~\rho(G)x_2=n_1x_1+(t-1)x_3,
\end{align*}
and $\rho(G)x_3=n_1x_1+n_2x_2+(t-2)x_3.$
It follows that
\begin{align}\label{align11}
x_1=\frac{\rho(G)+1}{\rho(G)+n_1}x_3~~\text{and}
~~x_2=\frac{\rho(G)+1}{\rho(G)+n_2}x_3.
\end{align}
Select $u_0\in V_1\setminus W$.
Let $G''$ be the graph obtained from $G$ by deleting edges
from $u_0$ to $V_2\setminus W$
and adding all edges from $u_0$ to $V_1\setminus (W\cup\{u_0\})$.
Then $G''\cong K_{t-1}+K_{n_1-1,n_2+1}$, and thus $G''$ is still $tC_{2\ell+1}$-free.
Moreover,
$$\rho(G'')-\rho(G)\ge
\sum_{v\in V_1\setminus (W\cup\{u_0\})}2x_{u_0}x_v-
   \sum_{v\in V_2\setminus W}2x_{u_0}x_v=2x_1\big((n_1-1)x_1-n_2x_2\big).
$$
In view of \eqref{align11}, we have
$$(n_1-1)x_1-n_2x_2=
\frac{(\rho(G)+1)\left((n_1-n_2-1)\rho(G)-n_2\right)}{(\rho(G)+n_1)(\rho(G)+n_2)}x_3
>0,
$$
since $n_1\geq n_2+2$ and $\rho(G)>\frac n2>n_2$.
It follows that $\rho(G'')>\rho(G),$ a contradiction.
Therefore, $n_1-n_2\le 1$ and $G\cong K_{t-1}+T_{n-t+1,2}$.
This completes the proof.
\end{proof}

\section{Proof of Theorem \ref{theorem1.5}}\label{section5}

In this section, we will often assume that
$n$ is sufficiently large without saying so explicitly.
We first give the lower and upper bounds of $\rho(S_{n,\ell}^{+})$ and $\rho(S_{n,\ell}^{++})$.

\begin{lem}\label{lem5.1}
For fixed $\ell$ and sufficiently large $n$, \\
(i) $\rho(S_{n,\ell}^{++})\ge \rho(S_{n,\ell}^+)\ge \frac{\ell-1+\sqrt{(\ell-1)^2+4\ell(n-\ell)}}{2}\ge\sqrt{\ell n}$ if $\ell\geq2$;\\
(ii) $\rho(S_{n,\ell}^{++})\le \sqrt{\big(\ell+\frac{1}{4\ell}\big)n}$
if $\ell\geq1$.
\end{lem}

\begin{proof}
(i) By a straightforward calculation, we have
$\rho(S_{n,\ell})=\frac{\ell-1+\sqrt{(\ell-1)^2+4\ell(n-\ell)}}2.$
Since $S_{n,\ell}\subseteq S_{n,\ell}^+\subseteq S_{n,\ell}^{++}$,
the inequality holds obviously for $\ell\geq2$.

(ii) By Perron-Frobenius theorem, there exists a positive unit eigenvector $X=(x_1,\ldots,x_n)^T$ corresponding to $\rho$, where $\rho=\rho(S_{n,\ell}^{++})$.
Let $W$ be the set of dominating vertices in $S_{n,\ell}^{++}$, and $\overline{W}=V(S_{n,\ell}^{++})\setminus W$.
Choose $u_0\in W$ and $v_0\in\overline{W}$ with $x_{u_0}=\max_{u\in W}x_u$
and $x_{v_0}=\max_{v\in \overline{W}}x_v$.
Note that $|W|=\ell$.
Then, $\rho x_{u_0}\leq (\ell-1)x_{u_0}+(n-\ell)x_{v_0}$
and $\rho x_{v_0}\leq \ell x_{u_0}+x_{v_0}$.
Combining these two inequalities,
we obtain $$(\rho-\ell+1)(\rho-1)\leq (n-\ell)\ell.$$
If $\rho>\sqrt{\big(\ell+\frac{1}{4\ell}\big)n}$,
then $(\rho-\ell+1)(\rho-1)>(n-\ell)\ell$, a contradiction.
Thus, $\rho\le\sqrt{\big(\ell+\frac{1}{4\ell}\big)n}$.
\end{proof}

Recall that $\ell\geq 2$ in Theorem \ref{theorem1.5}.
We shall proceed the proof by induction on $t$.
When $t=1$, the result holds immediately by \cite{Cioaba1,Nikiforov5,Zhai-3}.
In the following, we assume that $t\geq 2$.

For convenience, set $\lambda=\ell t-1$,
then $\lambda\ge 2\ell-1$.
Let $G$ be an extremal graph with respect to $spex(n,tC_{2\ell})$.
Clearly, $G$ is connected.
By Perron-Frobenius theorem, there exists a positive unit eigenvector $X=(x_1,\ldots,x_n)^T$ corresponding to $\rho(G)$.
Choose $u^*\in V(G)$ with $x_{u^*}=\max\{x_i~|~i=1,2,\dots,n\}$.
For a vertex $u$ and a positive integer $i$,
let $N_i(u)$ denote the set of vertices at distance $i$ from $u$ in $G$.
By the induction hypothesis, we obtain that for $n'$ sufficiently large,
\begin{align}\label{align12}
spex\big(n',(t-1)C_{2\ell}\big)=\left\{
\begin{array}{ll}
\rho(S_{n',\lambda-\ell}^{++})  & \hbox{if $\ell=2$,} \\
\rho(S_{n',\lambda-\ell}^{+})  & \hbox{if $\ell\ge 3$.}
\end{array}
\right.
\end{align}

We then show that for each $u\in V(G)$,
$G-\{u\}$ contains $t-1$ disjoint copies of $C_{2\ell}$
through Lemmas \ref{lem5.2} and \ref{lem5.3}.
This will be used to bound $\rho(G)$ in Lemma \ref{lem5.4},
to bound $\sum_{v\in V(G)}d_G^2(v)$ in Lemma \ref{lem5.5}
and to prove a key property in Lemma \ref{lem5.6}.

\begin{lem}\label{lem5.2}
Let $H$ be a graph on $n-1$ vertices.
Then $\rho(H)\ge \rho(K_1+H)-\frac{n-1}{\rho(K_1+H)}$.
\end{lem}

\begin{proof}
Let $\overline{u}$ be the dominating vertex of $K_1+H$
corresponding to $K_1$.
Set $\overline{\rho}=\rho(K_1+H)$ and
$Y=(y_u)$ be an eigenvector to $\overline{\rho}$.
Using the Rayleigh quotient gives
$$\overline{\rho}=\frac{2\sum_{uv\in E(K_1+H)}y_{u}y_{v}}{\sum_{u\in V(K_1+H)}y_{u}^2}
=\frac{2\sum_{uv\in E(H)}y_uy_v+2y_{\overline{u}}\sum_{u\in V(H)}y_u}{y_{\overline{u}}^2+\sum_{u\in V(H)}y_u^2}.$$
Since $\overline{\rho}y_{\overline{u}}=\sum_{u\in V(H)}y_u$,
we have $y_{\overline{u}}\sum_{u\in V(H)}y_u=\overline{\rho}y_{\overline{u}}^2
=\frac{1}{\overline{\rho}}\big(\sum_{u\in V(H)}y_u\big)^2.$
Thus,
$$ 2\sum_{uv\in E(H)}y_uy_v=\overline{\rho}\sum_{u\in V(H)}y_u^2-\overline{\rho}y_{\overline{u}}^2
=\overline{\rho}\sum_{u\in V(H)}y_u^2-\frac{1}{\overline{\rho}}\Big(\sum_{u\in V(H)}y_u\Big)^2.$$
From Cauchy-Schwarz inequality that
$\big(\sum_{u\in V(H)}y_u\big)^2\leq(n-1)\sum_{u\in V(H)}y_u^2$.
It follows that
$$\rho(H)\ge \frac{2\sum_{uv\in E(H)}y_uy_v}{\sum_{u\in V(H)}y_u^2}\ge \overline{\rho}- \frac{n-1}{\overline{\rho}},$$
as desired.
\end{proof}

\begin{lem}\label{lem5.3}
For every vertex $u\in V(G)$, $G-\{u\}$ contains $t-1$ disjoint $2\ell$-cycles.
\end{lem}

\begin{proof}
Suppose to the contrary that there exists a vertex $u$ such that $G-\{u\}$ is $(t-1)C_{2\ell}$-free. Then $\rho(G-\{u\})\le spex\big(n-1,(t-1)C_{2\ell}\big)$.
It follows from \eqref{align12} that
\begin{align}\label{align13}
\rho\big(G-\{u\}\big)\le\rho\big(S_{n-1,\lambda-\ell}^{++}\big),
\end{align}
as $\rho(S_{n-1,\lambda-\ell}^+)\leq\rho(S_{n-1,\lambda-\ell}^{++}).$

Recall that $t,\ell\ge2$ and $\lambda\geq2\ell-1\ge3$.
We can easily check that
$\sqrt{\lambda}-\frac{1}{\sqrt{\lambda}}>
\sqrt{\lambda-\ell+\frac{1}{4(\lambda-\ell)}}.$
By Lemma \ref{lem5.1} (ii), we further have
\begin{align}\label{align14}
  \sqrt{\lambda n}-\frac{n}{\sqrt{\lambda n}}> \sqrt{\Big(\lambda-\ell+\frac{1}{4(\lambda-\ell)}\Big)n}
\ge\rho(S_{n,\lambda-\ell}^{++})>\rho(S_{n-1,\lambda-\ell}^{++}).
\end{align}

On the one hand, $u$ is a dominating vertex of $G$.
Otherwise, there exists a vertex $v$ not adjacent to $u$.
Let $G^*$ be the graph obtained from $G$ by adding the edge $uv$.
Since $G^*-\{u\}=G-\{u\}$, $G^*-\{u\}$ is also $(t-1)C_{2\ell}$-free,
and thus $G^*$ is $tC_{2\ell}$-free.
However, $G\subset G^*$ indicates that $\rho(G)<\rho(G^*)$,
contradicting the fact that $G$ is extremal with respect to $spex(n,tC_{2\ell})$.

On the other hand, notice that $S_{n,\lambda}^+$ is $tC_{2\ell}$-free,
then $\rho(G)\ge \rho(S_{n,\lambda}^+)$, and so
$\rho(G)\ge \sqrt{\lambda n}$ by Lemma \ref{lem5.1} (i).
Since $u$ is a dominating vertex of $G$, one can see $G\cong K_1+(G-\{u\})$.
Combining $\rho(G)\ge \sqrt{\lambda n}$ and \eqref{align14}
with Lemma \ref{lem5.2}, we have
$$\rho\big(G-\{u\}\big)\ge \rho(G)-\frac{n-1}{\rho(G)}\geq
\sqrt{\lambda n}-\frac{n}{\sqrt{\lambda n}}
>\rho\big(S_{n-1,\lambda-\ell}^{++}\big),$$
which contradicts \eqref{align13}.
Therefore, the lemma holds.
\end{proof}

Choose $u\in V(G)$ arbitrarily.
By Lemma \ref{lem5.3},
$G-\{u\}$ contains $t-1$ disjoint ${2\ell}$-cycles,
say $C^1,\dots,C^{t-1}$.
Let $V'=\cup_{j=1}^{t-1}V(C^j)$ and $G'=G-V'$.
Then $G'$ is $C_{2\ell}$-free.
Set $N_i'(u)=N_i(u)\setminus V'$ for $i\in\{1,2\}$.
Clearly, $G'[N_1'(u)]$ is $P_{2\ell-1}$-free.
By Lemma \ref{lemma2.5}, $ e(N_1'(u))\le(\ell-\frac32)|N_1'(u)|\leq (\ell-\frac32)|N_1(u)|$ and so
\begin{eqnarray}\label{align15}
 e\big(N_1(u)\big) &\le&  e\big(N_1'(u)\big)+|N_1(u)\cap V'||N_1(u)|
 \le \Big(\ell-\frac32+2\ell(t-1)\Big)|N_1(u)|\nonumber\\
 &\le& \Big(2\lambda-\frac32\Big)|N_1(u)|,
\end{eqnarray}
as $\ell\geq2$ and $\lambda=\ell t-1$.
Furthermore, let $W_0$ be an arbitrary subset of $V(G)$.
Then, the bipartite subgraph $G[N_1'(u),N_2'(u)\cap W_0]$ is $P_{2\ell+1}$-free
(otherwise, we can find a $P_{2\ell-1}$ with both endpoints in $N_1'(u)$ and thus a $C_{2\ell}$ in $G'$).
By Lemma \ref{lemma2.5},
$e(N_1'(u),N_2'(u)\cap W_0)\le (\ell-\frac12)(|N_1(u)|+|N_2(u)\cap W_0|).$
Since both $N_1(u)\setminus N_1'(u)$
and $(N_2(u)\cap W_0)\setminus (N_2'(u)\cap W_0)$ are subsets of $V'$,
we have
\begin{eqnarray}\label{align16}
e\big(N_1(u),N_2(u)\cap W_0\big)
&\le& e\big(N_1'(u),N_2'(u)\cap W_0\big)+|V'|(|N_1(u)|+|N_2(u)\cap W_0|)\nonumber\\
&\le& \Big(\ell-\frac12+2\ell(t-1)\Big)(|N_1(u)|+|N_2(u)\cap W_0|)\nonumber\\
&\le& \Big(2\lambda-\frac12\Big)\big(|N_1(u)|+|N_2(u)\cap W_0|\big).
\end{eqnarray}

\begin{lem}\label{lem5.4}
$\sqrt{\lambda n}\le\rho(G)\leq\sqrt{6\lambda n}.$
\end{lem}

\begin{proof}
Recall that $S_{n,\lambda}^{+}$ is $tC_{2\ell}$-free.
Then $\rho(G)\geq\rho(S_{n,\lambda}^{+})$, and
the lower bound follows from Lemma \ref{lem5.1} (i).
We then prove the upper bound.
Note that
$$\rho^2(G)x_{u^*}=\!\!\sum_{u\in N_1(u^*)}\sum_{w\in N_1(u)}\!\!x_w
\le |N_1(u^*)|x_{u^*}+2e\big(N_1(u^*)\big)x_{u^*}+e\big(N_1(u^*),N_2(u^*)\big)x_{u^*}.$$
Setting $u=u^*$ and $W_0=N_2(u^*)$ in (\ref{align15}) and (\ref{align16}), we obtain
$\rho^2(G)\leq\big(6\lambda-\frac52\big)n\leq 6\lambda n.$
\end{proof}

In \cite{Nikiforov6}, Nikiforov studied an extremal problem on degree power, which is
an extension of Tur\'{a}n's problem. He showed that $\sum_{u\in V(H)}d_H^2(u)\le 2(\ell-1)e(H)+(\ell-1)(|H|-1)|H|$ for every $C_{2\ell}$-free graph $H$.
Inspired by this result, we obtain the following one on
$tC_{2\ell}$-free graphs.

\begin{lem}\label{lem5.5}
We have $e(G)\le \ell n^{1+\frac{1}{\ell}}$ and
$\sum_{v\in V(G)}d_G^2(v)<2\lambda n^2$.
\end{lem}

\begin{proof}
From the above definition of $G'$,
we know that $|G'|=n-2\ell(t-1)$ and $G'$ is $C_{2\ell}$-free.
By Lemma \ref{lemma2.6}, we have
$$e(G')\le ex\big(n-2\ell(t-1),C_{2\ell}\big)\le (\ell-1)(n-2\ell(t-1))^{1+\frac{1}{\ell}}+16(\ell-1)n.$$
It follows that
\begin{align}\label{align17}
e(G)\le e(G')+\sum_{v\in V'}d_{G}(v)\le e(G')+2\ell(t-1)n
\le \ell n^{1+\frac{1}{\ell}}.
\end{align}
Hence, the first statement holds. For an arbitrary vertex $u\in V(G)$,
$$\sum_{v\in N_1(u)}d_G(v)=|N_1(u)|+2e\big(N_1(u)\big)+e\big(N_1(u),N_2(u)\big).$$
Combining with \eqref{align15} and \eqref{align16},
where $W_0$ is defined as $N_2(u)$, we get that
$$\sum_{v\in N_1(u)}d_G(v)<(4\lambda-2)|N_1(u)|+\Big(2\lambda-\frac12\Big)n.$$
Summing the above inequality over all vertices $u\in V(G)$ and using \eqref{align17},
we obtain
\begin{eqnarray*}
\sum_{u\in V(G)}\sum_{v\in N_1(u)}d_G(v)
&<& (4\lambda-2)\sum_{u\in V(G)}d_G(u)+ \Big(2\lambda-\frac12\Big)n^2 \nonumber\\
&=&(8\lambda-4)e(G)+\Big(2\lambda-\frac12\Big)n^2\nonumber\\
&<& 2\lambda n^2.
\end{eqnarray*}
Observe that $\sum_{v\in V(G)}d_G^2(v)=\sum_{u\in V(G)}\sum_{v\in N_1(u)}d_G(v).$
The second statement follows.
\end{proof}

Choose a positive constant $\eta<\frac{1}{20000\lambda^5}$,
and define $W=\{u\in V(G)~|~x_u\ge \eta x_{u^*}\}.$
We shall give an upper bound for $|W|$ and a lower bound for
degrees of vertices in $W$ (see Lemmas \ref{lem5.8} and \ref{lem5.9}).
However, we are in trouble when we deal with the case $\ell=2$.
Hence, we have to prove an extra structural property
(see Lemma \ref{lem5.6}).

\begin{lem}\label{lem5.6}
For $\ell=2$, we have $\Delta(G)\ge (1-\frac{\eta}{40\lambda}) n$.
\end{lem}

\begin{proof}
Set $\alpha=1-\frac{\eta}{40\lambda}$
and suppose to the contrary that $\Delta(G)<\alpha n$.
Specially, we have $d_G(u^*)<\alpha n$.
By Lemma \ref{lem5.3},
$G-\{u^*\}$ contains $t-1$ disjoint quadrilaterals
$C^1,\dots,C^{t-1}$.
Given an arbitrary $i\in \{1,2,\dots,t-1\}$,
we assume that $V(C^i)=\{u_{ij}~|~j=1,2,3,4\}$.
We then define $M_i=\{u_{ij_1},u_{ij_2}\}$
if there exist two distinct vertices $u_{ij_1},u_{ij_2}\in V(C^i)$ with $|N_1(u_{ij_1})\cap N_1(u_{ij_2})|\geq(1-\alpha)n$,
and $M_i=V(C^i)$ otherwise.
Furthermore, let $M=\cup_{i=1}^{t-1}M_i$ and $G'=G-M$.

If $M_i=\{u_{ij_1},u_{ij_2}\}$ for some $u_{ij_1},u_{ij_2}\in V(C^i)$,
then $e\big(M_i,V(G')\big)\le d_G(u_{ij_1})+d_G(u_{ij_2})<2\alpha n$.
If $M_i=V(C^i)$, then
 $$e\big(M_i,V(G')\big)\le\Big|\cup_{j=1}^{4}N_1(u_{ij})\Big|+\!\!\sum_{1\le j_1<j_2\le 4}\Big|N_1(u_{ij_1})\cap N_1(u_{ij_2})\Big|<
 n+6(1-\alpha)n<2\alpha n,$$
which the last inequality follows from $a>\frac78.$
Hence, we always have
\begin{eqnarray}\label{align18}
e\big(M,V(G')\big)=
\sum_{i=1}^{t-1}e\big(M_i,V(G')\big)<2(t-1)\alpha n.
\end{eqnarray}

Recall that $G'=G-M=G-\cup_{i=1}^{t-1}M_i$.
We will see that $G'$ is $C_4$-free.
Otherwise, let ${\widetilde{C}}^t$ be a 4-cycle in $G'$.
If $M_1=V(C^1)$, then we define a 4-cycle ${\widetilde{C}}^1=C^1$,
where $V({\widetilde{C}}^1)\cap V({\widetilde{C}}^t)=\varnothing$ obiviously.
If $M_1=\{u_{1j_1},u_{1j_2}\}$ for some $u_{1j_1},u_{1j_2}\in V(C^1)$,
then $|N_1(u_{1j_1})\cap N_1(u_{1j_2})|\ge (1-\alpha)n$,
and thus there exists a 4-cycle ${\widetilde{C}}^1=u_{1j_1}v_1u_{1j_2}w_1u_{1j_1}$
such that $v_1,w_1\notin M\cup V({\widetilde{C}}^t).$
Similarly, if $M_2=V(C^2)$, then we define ${\widetilde{C}}^2=C^2$;
otherwise, $M_2=\{u_{2j_1},u_{2j_2}\}$, then we can find a 4-cycle ${\widetilde{C}}^2=u_{2j_1}v_2u_{2j_2}w_2u_{2j_1}$ such that
$v_2,w_2\notin M\cup V({\widetilde{C}}^1)\cup V({\widetilde{C}}^t).$
Repeating the above steps, we obtain a sequence of disjoint 4-cycles ${\widetilde{C}}^1,\cdots,{\widetilde{C}}^{t-1}$ such that
$V({\widetilde{C}}^i)\cap V(\widetilde{C}^{t})=\varnothing$ for $1\leq i\leq t-1$.
Consequently, we find $t$ disjoint 4-cycles in $G$, a contradiction.
Therefore, $G'$ is $C_4$-free.

We know that
\begin{eqnarray}\label{align19}
\rho^2(G)x_{u^*}=\sum_{v\in N_1(u^*)}\sum_{w\in N_1(v)}x_w.
\end{eqnarray}
In the following, we distinguish (\ref{align19})
into the following three cases.

{\bf Case (i) $v,w\in V(G')$.}
Let $\widetilde{N}_i(u^*)$ be the set of vertices at distance $i$ from $u^*$ in $G'$.
Then $\widetilde{N}_i(u^*)\subseteq N_i(u^*)$ for $i\in\{1,2\}$
as $G'$ is an induced subgraph of $G$.
We evaluate the following term.
$$\sum_{v\in \widetilde{N}_1(u^*)}
\sum_{w\in \widetilde{N}_1(v)}x_w=
\sum_{v\in \widetilde{N}_1(u^*)}
\Big(x_{u^*}+\!\!\!\!\sum_{w\in \widetilde{N}_1(v)\cap\widetilde{N}_1(u^*)}x_w
+\!\!\!\!\sum_{w\in \widetilde{N}_1(v)\cap\widetilde{N}_2(u^*)}x_w\Big).$$
Since $G'$ is $C_4$-free,
$G'[\widetilde{N}_1(u^*)]$ is $P_3$-free, that is, $\Delta\big(G'[\widetilde{N}_1(u^*)]\big)\leq1$.
Thus
$$\sum_{v\in \widetilde{N}_1(u^*)}\sum_{w\in \widetilde{N}_1(v)\cap\widetilde{N}_1(u^*)}x_w
\leq\sum_{v\in \widetilde{N}_1(u^*)}x_v
\leq\sum_{v\in N_1(u^*)}x_v=\rho(G)x_{u^*}.$$
Moreover, any two vertices in $\widetilde{N}_1(u^*)$ have no common neighbors in $\widetilde{N}_2(u^*)$,
which implies that
$e\big(\widetilde{N}_1(u^*),\widetilde{N}_2(u^*)\big)=|\widetilde{N}_2(u^*)|$.
It follows that
$$ \sum_{v\in \widetilde{N}_1(u^*)}\sum_{w\in \widetilde{N}_1(v)\cap\widetilde{N}_2(u^*)}x_w\leq
e\big(\widetilde{N}_1(u^*),\widetilde{N}_2(u^*)\big)x_{u^*}
=|\widetilde{N}_2(u^*)|x_{u^*}.$$
Combining above three inequalities, we obtain
\begin{eqnarray}\label{align20}
\sum_{v\in \widetilde{N}_1(u^*)}
\sum_{w\in \widetilde{N}_1(v)}x_w\leq
\big(|\widetilde{N}_1(u^*)|+\rho(G)+|\widetilde{N}_2(u^*)|\big)x_{u^*}
<(\rho(G)+n-|M|\big)x_{u^*}.
\end{eqnarray}

{\bf Case (ii) $v\in V(G)$ and $w\in M$.}
We shall evaluate
$\sum_{v\in N_1(u^*)}\sum_{w\in N_1(v)\cap M}x_w.$
Note that $N_1(u^*)\setminus M=\widetilde{N}_1(u^*).$
On the one hand, $\sum_{v\in N_1(u^*)\setminus M}\sum_{w\in N_1(v)\cap M}x_w
\leq e\big(\widetilde{N}_1(u^*),M\big)x_{u^*}.$
On the other hand,
\begin{eqnarray*}
\sum_{v\in N_1(u^*)\cap M}\sum_{w\in N_1(v)\cap M}x_w
\le 2e(M)x_{u^*}\le 2\binom{|M|}{2}x_{u^*}<
(1-\alpha)n x_{u^*}.
\end{eqnarray*}
Thus we have
\begin{eqnarray}\label{align21}
\sum_{v\in N_1(u^*)}\sum_{w\in N_1(v)\cap M}x_w
<\Big(e\big(\widetilde{N}_1(u^*),M\big)+(1-\alpha)n\Big)x_{u^*}.
\end{eqnarray}

{\bf Case (iii) $v\in M$ and $w\in V(G')$.}
We shall calculate
$\sum_{v\in N_1(u^*)\cap M}\sum_{w\in N_1(v)\cap V(G')}x_w.$
Now set $\widetilde{N}_{2^+}(u^*)=V(G')
\setminus\big(\{u^*\}\cup\widetilde{N}_1(u^*)\big).$
We can observe that
\begin{eqnarray}\label{align22}
\sum_{v\in N_1(u^*)\cap M}\sum_{w\in N_1(v)\cap V(G')}x_w
&=&\sum_{v\in N_1(u^*)\cap M}
\Big(x_{u^*}+\!\!\!\!\sum_{w\in N_1(v)\cap\widetilde{N}_1(u^*)}\!\!\!\!x_w
+\!\!\!\!\sum_{w\in N_1(v)\cap\widetilde{N}_{2^+}(u^*)}\!\!\!\!x_w\Big)\nonumber\\
&\leq& \sum_{v\in M}
\Big(x_{u^*}+\!\!\!\!\sum_{w\in\widetilde{N}_1(u^*)}\!\!\!\!x_w
+\!\!\!\!\sum_{w\in N_1(v)\cap\widetilde{N}_{2^+}(u^*)}\!\!\!\!x_w\Big)\nonumber\\
&\leq& |M|\big(1+\rho(G)\big)x_{u^*}+e\big(\widetilde{N}_{2^+}(u^*),M\big)x_{u^*}.
\end{eqnarray}
Summing (\ref{align20}), (\ref{align21}) and (\ref{align22}) into
(\ref{align19}),
we obtain
$$\rho^2(G)<(|M|+1)\rho(G)+(2-\alpha)n+e\big(V(G'),M\big).$$
Note that $|M|=4(t-1)$, and by Lemma \ref{lem5.4},
$\rho(G)\leq\sqrt{6\lambda n}$. Thus,
$(|M|+1)\rho(G)\leq(4t-3)\sqrt{6\lambda n}<(1-\alpha)n$.
Moreover, $e\big(V(G'),M\big)<2(t-1)\alpha n$ by (\ref{align18}).
Consequently, $$\rho^2(G)<\big(3+(2t-4)\alpha\big)n\leq(2t-1)n,$$
as $t\geq2$ and $\alpha=1-\frac{\eta}{40\lambda}<1$.
However, by Lemma \ref{lem5.4}
we have $\rho^2(G)\geq \lambda n=(\ell t-1)n\geq(2t-1)n,$
a contradiction,
Therefore, $\Delta(G)\ge\alpha n$, completing the proof.
\end{proof}

\begin{lem}\label{lem5.7}
Let $W'=\{u\in V(G)~|~x_u\ge \frac{\eta}5x_{u^*}\}$.
Then $|W'|\le\frac{\eta}{20\lambda}n$.
\end{lem}

\begin{proof}
We first consider the case $\ell\ge 3$.
By Lemma \ref{lem5.4}, $\rho(G)\geq\sqrt{\lambda n}$.
Hence,
$$\sqrt{\lambda n}\frac{\eta}5x_{u^*}\le \sqrt{\lambda n}x_u
\leq\rho(G) x_u=\sum_{v\in N_1(u)}x_v\le d_G(u)x_{u^*}$$
for each $u\in W'$.
Summing this inequality over all vertices $u\in W'$, we obtain
\begin{align}\label{align23}
|W'|\sqrt{\lambda n}\frac{\eta}5 x_{u^*}\le \sum_{u\in W'}d_G(u)x_{u^*}
\le \sum_{u\in V(G)}d_G(u)x_{u^*}
\le 2e(G)x_{u^*}.
\end{align}
Combining \eqref{align17} and \eqref{align23}, we get
$|W'|\le \frac{10\ell n^{1+\frac{1}{\ell}}}{\sqrt{\lambda n}\eta}\le \frac{\eta}{20\lambda}n$
for $n$ large enough.

Now, it remains the case $\ell=2$.
By Lemma \ref{lem5.6}, there exists a vertex $v^*\in V(G)$ with $d_G(v^*)\ge(1-\frac{\eta}{40\lambda}) n$.
Hence, $$|W'\setminus N_1(v^*)|\leq |V(G)\setminus N_1(v^*)|=n-d_G(v^*)\leq\frac{\eta}{40\lambda}n.$$
Let $W^*=\{v\in N_1(v^*)~|~x_v\ge\sqrt{6\lambda}n^{-0.4}x_{u^*}\}$.
Note that $\rho(G)\leq\sqrt{6\lambda n}$ by Lemma \ref{lem5.4}.
Thus
$$\big|W^*\big|\sqrt{6\lambda}n^{-0.4}x_{u^*}
\le\sum_{v\in N_1(v^*)}x_v=\rho(G)x_{v^*}\le\sqrt{6\lambda n}x_{u^*},$$
yielding $|W^*|\le n^{0.9}\leq\frac{\eta}{40\lambda}n$.
Since $\frac{\eta}5x_{u^*}>\sqrt{6\lambda}n^{-0.4}x_{u^*}$,
we have $W'\cap N_1(v^*)\subseteq W^*$, and so
$|W'\cap N_1(v^*)|\leq|W^*|\leq\frac{\eta}{40\lambda}n$.
Combining $|W'\setminus N_1(v^*)|\leq \frac{\eta}{40\lambda}n$
gives
$|W'|\le\frac{\eta}{20\lambda}n,$
as claimed.
\end{proof}

\begin{lem}\label{lem5.8}
$|W|\leq\frac{128\lambda^3}{\eta^2}$.
\end{lem}

\begin{proof}
We first prove that $d_G(u)>\frac{\eta}{8\lambda}n$ for each $u\in W$.
Suppose to the contrary that there exists a vertex $\widetilde{u}\in W$ with $d_G(\widetilde{u})\leq\frac{\eta}{8\lambda}n$.
Then $x_{\widetilde{u}}\ge \eta x_{u^*}$ as $\widetilde{u}\in W$,
and by Lemma \ref{lem5.4} $\rho(G)\ge\sqrt{\lambda n}$.
Thus we have
\begin{align}\label{align24}
\eta\lambda n x_{u^*}\le\rho^2(G) x_{\widetilde{u}}= |N_1(\widetilde{u})|x_{\widetilde{u}}+\sum_{u\in N_1(\widetilde{u})}\!\!\!d_{N_1(\widetilde{u})}(u)x_u
+\!\!\!\sum_{u\in N_2(\widetilde{u})}\!\!\!d_{N_1(\widetilde{u})}(u)x_u.
\end{align}
In view of \eqref{align15},
we have $e(N_1(\widetilde{u}))\leq(2\lambda-\frac32)d_G(\widetilde{u})$.
Note that $|N_1(\widetilde{u})|\leq\frac{\eta}{8\lambda}n$. Thus,
$$|N_1(\widetilde{u})|x_{\widetilde{u}}+\!\!\!\sum_{u\in N_1(\widetilde{u})}\!\!\!d_{N_1(\widetilde{u})}(u)x_u
\leq \big(|N_1(\widetilde{u})|+2e(N_1(\widetilde{u}))\big)x_{u^*}
\leq(4\lambda-2)|N_1(\widetilde{u})|x_{u^*}\leq \frac{1}{2}\eta n x_{u^*}.$$
Combining the above inequality with \eqref{align24}, we obtain
\begin{eqnarray}\label{align25}
\sum_{u\in N_2(\widetilde{u})}d_{N_1(\widetilde{u})}(u)x_u
\ge(\lambda-\frac{1}{2})\eta n x_{u^*}.
\end{eqnarray}
Now, setting $u=\widetilde{u}$ and $W_0=W'$ in (\ref{align16}),
we have
$$e\big(N_1(\widetilde{u}), N_2(\widetilde{u})\cap W'\big)
\leq\big(2\lambda-\frac12\big)\big(|N_1(\widetilde{u})|
+|N_2(\widetilde{u})\cap W'|\big)\leq 2\lambda\big(|N_1(\widetilde{u})|
+|W'|\big).$$
Since $|N_1(\widetilde{u})|\leq\frac{\eta}{8\lambda}n$,
and $|W'|\le\frac{\eta}{20\lambda}n$ by Lemma \ref{lem5.7},
it follows that
\begin{eqnarray}\label{align26}
\sum_{u\in N_2(\widetilde{u})\cap W'}d_{N_1(\widetilde{u})}(u)x_u
\leq e\big(N_1(\widetilde{u}), N_2(\widetilde{u})\cap W'\big)x_{u^*}
\leq\frac{7}{20}\eta nx_{u^*}.
\end{eqnarray}
Note that $x_u<\frac{\eta}{5}x_{u^*}$ for each $u\in V(G)\setminus W'$.
Setting $u=\widetilde{u}$ and $W_0=V(G)\setminus W'$ in (\ref{align16}),
we get $e\big(N_1(\widetilde{u}),N_2(\widetilde{u})\setminus W'\big)\leq
(2\lambda-\frac12)n.$ Consequently,
$$\sum_{u\in N_2(\widetilde{u})\setminus W'}d_{N_1(\widetilde{u})}(u)x_u
\leq e\big(N_1(\widetilde{u}),N_2(\widetilde{u})\setminus W'\big)\frac{\eta}{5}x_{u^*}
\leq(2\lambda-\frac12)n\frac{\eta}{5}x_{u^*}.$$
Combining \eqref{align26} gives
$$\sum_{u\in N_2(\widetilde{u})}d_{N_1(\widetilde{u})}(u)x_u
\leq\Big(\frac{7}{20}+\frac{4\lambda-1}{10}\Big)\eta nx_{u^*}
<(\lambda-\frac12)\eta n x_{u^*}$$
as $\lambda=\ell t-1>\frac54$,
contradicting \eqref{align25}.
Therefore, $d_G(u)>\frac{\eta}{8\lambda}n$ for each $u\in W$.
It follows that $\sum_{u\in V(G)}d_G^2(u)\geq\sum_{u\in W}d_{G}^2(u)\geq |W|\big(\frac{\eta}{8\lambda}n\big)^2$.
Moreover, $\sum_{u\in V(G)}d_G^2(u)<2\lambda n^2$ by Lemma \ref{lem5.5}.
Thus, $|W|\leq\frac{128\lambda^3}{\eta^2}$,
as claimed.
\end{proof}

\begin{lem}\label{lem5.9}
For each $u\in W$,
we have $d_G(u)\ge\big(\frac{x_u}{x_{u^*}}-20\eta\big)n$.
\end{lem}

\begin{proof}
Let $u$ be an arbitrary vertex in $W$.
For convenience,
we use $W_i$ and $\overline{W_i}$
instead of $N_i(u)\cap W$ and $N_i(u)\setminus W$, respectively.
In view of \eqref{align15} and \eqref{align16}, we have
$\max\{e(N_1(u)),e(N_1(u),N_2(u))\}\leq 2\lambda n$.
Since $W_i\cup\overline{W_i}=N_i(u)$ for $i\in\{1,2\}$, we can see that
\begin{eqnarray}\label{align27}
\max\{e(\overline{W_1}),e(W_1,\overline{W_1}),
e(W_1,\overline{W_2}),e(\overline{W_1},\overline{W_2})\}
\le2\lambda n.
\end{eqnarray}
Recall that $\rho(G)\geq\sqrt{\lambda n}$. We also have
\begin{eqnarray}\label{align28}
\lambda n x_{u}\le \rho^2(G)x_{u}=\sum_{v\in N_1(u)}\sum_{w\in N_1(v)}x_w=|N_1(u)|x_u+\sum_{v\in N_1(u)}\sum_{w\in N_1(v)\setminus\{u\}}x_w.
\end{eqnarray}
Note that $N_1(u)=W_1\cup\overline{W_1}$
and for any $v\in N_1(u)$,
$$N_1(v)\setminus\{u\}=N_1(v)\cap\big(N_1(u)\cup N_2(u)\big)=N_1(v)\cap(W_1\cup \overline{W_1}\cup W_2\cup\overline{W_2}).$$

We now calculate the term
$\sum_{v\in N_1(u)}\sum_{w\in N_1(v)\setminus\{u\}}x_w$ in (\ref{align28}).
We first consider the case $v\in W_1$.
Note that $x_w\le x_{u^*}$ for $w\in W_1\cup W_2$
and $x_w\le\eta x_{u^*}$ for $w\in\overline{W_1}\cup\overline{W_2}$.
Thus,
$$\sum_{v\in W_1}\sum_{w\in N_1(v)\setminus \{u\}}\!\!\!x_w\le
\big(2e(W_1)+e(W_1,W_2)\big)x_{u^*}+\big(e(W_1,\overline{W_1})
+e(W_1,\overline{W_2})\big)\eta x_{u^*}.$$
On the one hand,
$|W|<\frac{128\lambda^3}{\eta^2}$ by Lemma \ref{lem5.8}.
Note that $W_1\cup W_2\subseteq W$. Thus,
$2e(W_1)+e(W_1,W_2)\le
2\binom{|W|}{2}\le\eta\lambda n.$
On the other hand, we have
$e(W_1,\overline{W_1})+e(W_1,\overline{W_2})
\le4\lambda n$ by \eqref{align27}.
Therefore,
\begin{eqnarray}\label{align29}
 \sum_{v\in W_1}\sum_{w\in N_1(v)\setminus \{u\}}x_w\le 5\lambda \eta n x_{u^*}.
\end{eqnarray}

Now, we consider the case $v\in \overline{W_1}$.
We can see that
\begin{eqnarray}\label{align30}
\sum_{v\in\overline{W_1}}\sum_{w\in N_1(v)\setminus\{u\}}x_w
&\le& \sum_{v\in\overline{W_1}}\sum_{w\in N_1(v)\cap(W_1\cup W_2)}x_w+
\sum_{v\in\overline{W_1}}\sum_{w\in N_1(v)\cap(\overline{W_1}\cup\overline{W_2})}x_w\nonumber\\
&\le& e(\overline{W_1},W_1\cup W_2)x_{u^*}+
\big(2e(\overline{W_1})+e(\overline{W_1},\overline{W_2})\big)\eta x_{u^*}\nonumber\\
&\le& e(\overline{W_1},W_1\cup W_2)x_{u^*}+
6\lambda\eta nx_{u^*},
\end{eqnarray}
where the last inequality follows from \eqref{align27}.
In the following, we shall evaluate $e(\overline{W_1},W_1\cup W_2)$.
Since $W_1\cup W_2\subseteq W$, it suffices to calculate $e(\overline{W_1},W)$.

Let $\overline{W_1}'$ be the subset of $\overline{W_1}$
in which each vertex has at least $\lambda$
neighbors in $W_1\cup W_2$.
If $|W_1\cup W_2|\le\lambda-1$, then $|\overline{W_1}'|=0$.
If $|W_1\cup W_2|\ge\lambda$, then we claim that
$|\overline{W_1}'|<(\lambda+1)\binom{|W_1\cup W_2|}{\lambda}$.
Otherwise, since there are only $\binom{|W_1\cup W_2|}{\lambda}$ options for all vertices in $\overline{W_1}'$ to choose a set of $\lambda$ neighbors from $W_1\cup W_2$,
we can find $\lambda$ vertices in $W_1\cup W_2$ with at least $|\overline{W_1}'|/\binom{|W_1\cup W_2|}{\lambda}\ge\lambda+1$ common neighbors in $\overline{W_1}'$. Moreover,
note that $u\notin W_1\cup W_2$
and $\overline{W_1}'\subseteq\overline{W_1}\subseteq N_1(u)$.
Hence, $G$ contains a copy of $K_{\lambda+1,\lambda+1}$, and thus $t$ disjoint
$2\ell$-cycles, a contradiction.
Therefore, we always have $|\overline{W_1}'|<(\lambda+1)\binom{|W_1\cup W_2|}{\lambda}\leq (\lambda+1)\binom{|W|}{\lambda}$.

By Lemma \ref{lem5.8}, $|W|$ is constant.
Now $|\overline{W_1}'|$ is also constant.
Thus, $|\overline{W_1}'||W_1\cup W_2|\leq9\lambda\eta n$.
Moreover, from the definition of $\overline{W_1}'$
we know $e(\overline{W_1}\setminus \overline{W_1}',W_1\cup W_2)\leq
(\lambda-1)|\overline{W_1}\setminus \overline{W_1}'|.$
Thus
\begin{eqnarray}\label{align31}
e(\overline{W_1},W_1\cup W_2)\leq
e(\overline{W_1}',W_1\cup W_2)+e(\overline{W_1}\setminus \overline{W_1}',W_1\cup W_2)
\le9\lambda\eta n+(\lambda-1)|N_1(u)|.
\end{eqnarray}
Back to (\ref{align30}), we obtain
$\sum_{v\in\overline{W_1}}\sum_{w\in N_1(v)\setminus\{u\}}x_w
\leq\big(15\lambda\eta n+(\lambda-1)|N_1(u)|\big)x_{u^*}.$
Combining this with (\ref{align28}) and (\ref{align29}),
we get that
$$\lambda n x_u
\leq|N_1(u)|x_{u}+20\lambda\eta nx_{u^*}+
(\lambda-1)|N_1(u)|x_{u^*}\leq\big(20\lambda\eta n+\lambda|N_1(u)|\big)x_{u^*},$$
which yields $|N_1(u)|\ge \big(\frac{x_u}{x_{u^*}}-20\eta\big)n$,
as desired.
\end{proof}

Now, we define $W''=\{u\in V(G)~|~x_u\ge 5000\lambda^4\eta x_{u^*}\}$.
Recall that $\eta<\frac{1}{20000\lambda^5}$ and
$W=\{u\in V(G)~|~x_u\ge \eta x_{u^*}\}$.
Clearly, $u^*\in W''$ and $W''\subseteq W.$

\begin{lem}\label{lem5.10}
For every $v\in W''$, we have $x_v\ge (1-\frac{1}{200\lambda^3})x_{u^*}$
and $d_G(v)\ge (1-\frac{1}{100\lambda^3})n$.
Moreover, we have $|W''|=\lambda$.
\end{lem}

\begin{proof}
Suppose to the contrary that there exists $v_0\in W''$ with $x_{v_0}<(1-\frac{1}{200\lambda^3})x_{u^*}$.
We use $W_i$ and $\overline{W_i}$ to denote $N_i(u^*)\cap W$ and $N_i(u^*)\setminus W$, respectively.
We first prove that
$|\overline{W_1}\cap N_1(v_0)|\ge 4000\lambda^4\eta n$.
By Lemma \ref{lem5.9}, we have
$$|N_1(u^*)|\ge(1-20\eta)n~~~\text{and}
~~~|N_1(v_0)|\ge(5000\lambda^4\eta-20\eta)n,$$
as $x_{v_0}\ge 5000\lambda^4\eta nx_{u^*}$.
Moreover, by Lemma \ref{lem5.8}, we have
$|W|\leq\frac{128\lambda^3}{\eta^2}\leq 10\eta n$.
Hence, $|\overline{W_1}|=|N_1(u^*)\setminus W|\ge(1-30\eta)n,$
and so
\begin{eqnarray}\label{align32}
\big|\overline{W_1}\cap N_1(v_0)\big|\ge\big|\overline{W_1}\big|
+\big|N_1(v_0)\big|-n\ge(5000\lambda^4\eta-50\eta)n>4000\lambda^4\eta n.
\end{eqnarray}

In view of (\ref{align32}), $v_0$ has neighbors in $\overline{W_1}$.
Then $v_0$ is of distance at most two from $u^*$,
that is, $v_0\in N_1(u^*)\cup N_2(u^*)$.
Note that $v_0\in W''\subseteq W$. Thus, $v_0\in W_1\cup W_2$.
Recall that
$x_{v_0}<(1-\frac{1}{200\lambda^3})x_{u^*}$.
Now, setting $u=u^*$ in \eqref{align28}-\eqref{align30},
we can observe that
\begin{eqnarray*}
\lambda nx_{u^*}
\!&\leq&\!|N_1(u^*)|x_{u^*}+11\lambda\eta nx_{u^*}
+e\big(\overline{W_1},(W_1\cup W_2)\setminus\{v_0\}\big)x_{u^*}
+e\big(\overline{W_1},\{v_0\}\big)x_{v_0}\\
\!&<&\!|N_1(u^*)|x_{u^*}+11\lambda\eta nx_{u^*}
+e\big(\overline{W_1},(W_1\cup W_2)\big)x_{u^*}
-e\big(\overline{W_1},\{v_0\}\big)\frac{x_{u^*}}{200\lambda^3}.
\end{eqnarray*}
From (\ref{align31}) we know that
$e(\overline{W_1},W_1\cup W_2)\leq
9\lambda\eta n+(\lambda-1)|N_1(u^*)|.$
Thus,
$$\lambda n
\leq \lambda|N_1(u^*)|+20\lambda\eta n
-\frac{e\big(\overline{W_1},\{v_0\}\big)}{200\lambda^3}
<\lambda n+20\lambda\eta n
-\frac{e\big(\overline{W_1},\{v_0\}\big)}{200\lambda^3}.$$
Consequently,
$e\big(\overline{W_1},\{v_0\}\big)<4000\lambda^4\eta n,$
contradicting (\ref{align32}).
Thus $x_v\ge (1-\frac{1}{200\lambda^3})x_{u^*}$
for $v\in W''$.

Recall that $\eta<\frac{1}{20000\lambda^5}$.
Then by Lemma \ref{lem5.9}, we can see that for each $v\in W''$,
$$d_G(v)\geq\big(\frac{x_v}{x_{u^*}}-20\eta\big)n\ge
\Big(1-\frac{1}{200\lambda^3}-20\eta\Big)n
\ge\Big(1-\frac{1}{100\lambda^3}\Big)n.$$

It remains to show $|W''|=\lambda$.
We first suppose that $|W''|\ge\lambda+1$.
Note that every $v\in W''$ has at most $\frac{n}{100\lambda^3}$ non-neighbors.
It follows that any $\lambda+1$ vertices in $W''$ have at least $n-\frac{(\lambda+1)n}{100\lambda^3}\geq \lambda+1$ common neighbors.
Thus, $G$ contains $K_{\lambda+1,\lambda+1}$ as a subgraph.
Recall that $\lambda=\ell t-1$. Thus $G$ also contains $tC_{2\ell}$, a contradiction.
Therefore, $|W''|\leq\lambda$.

Next, suppose that $|W''|\le \lambda-1$.
Since $u^*\in W''\setminus (W_1\cup W_2)$,
we have $|W''\cap(W_1\cup W_2)|\le\lambda-2$.
Moreover, recall that $W_i\cup\overline{W_i}=N_i(u^*)$ for $i\in\{1,2\}$,
then
\begin{eqnarray*}
e\big(\overline{W_1},(W_1\cup W_2)\setminus W''\big)\leq e\big(\overline{W_1},W_1\big)+e\big(\overline{W_1},W_2\big)
\leq e(N_1(u^*))+e(N_1(u^*),N_2(u^*)).
\end{eqnarray*}
Setting $u=u^*$ and $W_0=N_2(u^*)$ in \eqref{align15} and \eqref{align16},
we have $e(N_1(u^*))+e(N_1(u^*),N_2(u^*))\le(4\lambda-2)n,$
and thus $e(\overline{W_1},(W_1\cup W_2)\setminus W'')\leq(4\lambda-2)n.$
Furthermore, we shall note that $x_w<5000\lambda^4\eta x_{u^*}$
for each $w\in(W_1\cup W_2)\setminus W''.$
Now, setting $u=u^*$ in \eqref{align28}-\eqref{align30}
and dividing both sides of \eqref{align28} by $x_{u^*}$,
we can see that
\begin{eqnarray*}
\lambda n
\!&\leq&\!|N_1(u^*)|\!+\!11\lambda\eta n\!+\!
e\big(\overline{W_1},(W_1\cup W_2)\cap W''\big)
\!+\!e\big(\overline{W_1},(W_1\cup W_2)\setminus W''\big)5000\lambda^4\eta\\
\!&\leq&\! n+11\lambda\eta n+(\lambda-2)n+5000\lambda^4\eta(4\lambda-2)n\\
\!&<&\! \lambda n,
\end{eqnarray*}
as $\eta<\frac{1}{20000\lambda^5}$.
This gives a contradiction.
Therefore, $|W''|=\lambda$.
\end{proof}

In the following, we complete the proof of Theorem \ref{theorem1.5}.

\begin{proof}
By Lemma \ref{lem5.10}, we see that
$|W''|=\lambda=\ell t-1$ and every vertex in $W''$ has
at most $\frac{n}{100\lambda^3}$ non-neighbors.
Now, let $U$ be the subset of $V(G)\setminus W''$
in which every vertex is a non-neighbor of some vertex in $W''$
and $U'=V(G)\setminus (W''\cup U)$.
Then, $G[W'',U']\cong K_{|W''|,|U'|}$.
Note that $|U|\le|W''|\frac{n}{100\lambda^3}=\frac{n}{100\lambda^2}$,
and thus
$|U'|\geq n-\lambda-\frac{n}{100\lambda^2}\geq \frac n2.$

We will see that $U=\varnothing$.
Suppose to the contrary that $U\neq\varnothing$.
Given $u\in U$ arbitrarily. Then,
$u$ has at most one neighbor in $U'$
(otherwise, $u$ has two neighbors $w_1,w_2\in
U'$, then $G[W'',U']$ together with $uw_1,uw_2$
gives a copy of $tC_{2\ell}$, a contradiction).
Moreover, by the definition of $U$,
$|N_1(u)\cap W''|\le|W''|-1=\lambda-1$.
It follows that
\begin{eqnarray}\label{align33}
\sum_{w\in N_1(u)\cap(W''\cup U')}\!\!\!\!x_w
=\!\!\sum_{w\in N_1(u)\cap W''}\!\!x_w+\!\!
\sum_{w\in N_1(u)\cap U'}\!\!x_w
\leq(\lambda-1)x_{u^*}+5000\lambda^4\eta x_{u^*}.
\end{eqnarray}

We now claim that $\rho(G)x_u\ge(\lambda-\frac{1}{200\lambda^2})x_{u^*}$.
Otherwise, let $G^*$ be the graph obtained from $G$
by deleting all edges incident to $u$ and joining $u$ to all vertices in $W''$.
Note that $|U'|\geq \frac n2$ and $N_{G^*}(u)\subseteq N_{G^*}(v)$ for any $v\in U'$.
Then $G^*$ is $tC_{2\ell}$-free (otherwise, $G^*-\{u\}$ contains $tC_{2\ell}$,
and thus $G-\{u\}$ too, a contradiction).
Moreover,
$$\rho(G^*)-\rho(G)\geq X^T\big(A(G^*)-A(G)\big)X
=2x_u\Big(\sum_{w\in W''}x_w-\sum_{w\in N_1(u)}x_w\Big).$$
Note that $\sum_{w\in W''}x_w\geq|W''|(1-\frac{1}{200\lambda^3})x_{u^*}
=(\lambda-\frac{1}{200\lambda^2})x_{u^*}$ by Lemma \ref{lem5.10}, but
$\sum_{w\in N_1(u)}x_w=\rho(G)x_u<(\lambda-\frac{1}{200\lambda^2})x_{u^*}$
by assumption.
Thus, $\rho(G^*)>\rho(G)$, a contradiction.

Now we have
$$\Big(\lambda-\frac{1}{200\lambda^2}\Big)x_{u^*}\leq\rho(G)x_u
=\!\!\sum_{w\in N_1(u)\cap(W''\cup U')}\!\!x_w+\!\!\sum_{w\in N_1(u)\cap U}\!\!x_w.$$
Combining \eqref{align33} gives
$$\frac{\sum_{w\in N_1(u)\cap U}x_w}{\rho(G)x_u}
\ge\frac{(\lambda-\frac{1}{200\lambda^2})x_{u^*}-
(\lambda-1+5000\lambda^4\eta)x_{u^*}}{(\lambda-\frac{1}{200\lambda^2})x_{u^*}}
\ge\frac{4}{5\lambda},$$
as $\eta<\frac{1}{20000\lambda^5}
<\frac{1}{5000\lambda^4}(\frac15-\frac{1}{200\lambda^2}
+\frac{1}{250\lambda^3})$.
Thus, $\sum_{w\in N_1(u)\cap U}x_w\ge\frac{4}{5\lambda}\rho(G)x_u$.

Now consider the matrix $A'=A(G[U])$ and the vector
$X'=X|_U$ (the restriction of $X$ to $U$).
We can observe that
$$(A'X')_u=\sum_{w\in N_1(u)\cap U}x_w
\ge\frac{4}{5\lambda}\rho(G)x_u$$
for each $u\in U$.
Since $X$ is the Perron vector of $G$,
$X'$ is a positive vector and thus
$A'X'\ge\frac{4}{5\lambda}\rho(G)X'$ entrywise.
Moreover, $\rho(G)\geq\sqrt{\lambda n}$ by Lemma \ref{lem5.4}.
Hence,
$$\rho(G[U])\ge \frac{X'^TA'X'}{X'^TX'}\ge\frac{4}{5\lambda}\rho(G)\ge \frac{4}{5}\sqrt{\frac n\lambda},$$
which also implies that $|U|=\Omega(\sqrt{n}).$
Since $G[U]$ is $tC_{2\ell}$-free,
we have $\rho(G[U])\le\sqrt{6\lambda|U|}$ by Lemma \ref{lem5.4}.
Recall that $|U|\le\frac{n}{100\lambda^2}$.
It follows that
$$\rho(G[U])\le\sqrt{\frac{6\lambda n}{100\lambda^2}}
<\frac{4}{5}\sqrt{\frac n\lambda},$$
a contradiction. Therefore, $U=\varnothing$.

\begin{figure}[!ht]
\centering
\begin{tikzpicture}[x=1.00mm, y=0.70mm, inner xsep=0pt, inner ysep=0pt, outer xsep=0pt, outer ysep=0pt]
\path[line width=0mm] (42.93,46.62) rectangle +(150.52,45.29);
\definecolor{L}{rgb}{0,0,0}
\path[line width=0.30mm, draw=L] (79.07,83.13) ellipse (19.11mm and 6.57mm);
\definecolor{F}{rgb}{0,0,0}
\path[line width=0.30mm, draw=L, fill=F] (65.52,81.79) circle (1.00mm);
\path[line width=0.30mm, draw=L, fill=F] (71.74,81.95) circle (1.00mm);
\path[line width=0.30mm, draw=L, fill=F] (81.84,81.95) circle (1.00mm);
\path[line width=0.30mm, draw=L, fill=F] (89.25,81.95) circle (1.00mm);
\path[line width=0.30mm, draw=L] (80.41,55.69) ellipse (22.81mm and 7.07mm);
\path[line width=0.30mm, draw=L, fill=F] (62.32,57.04) circle (1.00mm);
\path[line width=0.30mm, draw=L, fill=F] (68.04,57.21) circle (1.00mm);
\path[line width=0.30mm, draw=L, fill=F] (73.26,57.38) circle (1.00mm);
\path[line width=0.30mm, draw=L, fill=F] (80.50,57.21) circle (1.00mm);
\path[line width=0.30mm, draw=L, fill=F] (90.77,57.21) circle (1.00mm);
\path[line width=0.30mm, draw=L, fill=F] (97.84,57.38) circle (1.00mm);
\path[line width=0.30mm, draw=L] (62.49,57.21) -- (68.04,57.21);
\path[line width=0.30mm, draw=L] (68.21,57.21) -- (73.26,57.38);
\path[line width=0.30mm, draw=L, fill=F] (74.77,81.62) circle (0.40mm);
\path[line width=0.30mm, draw=L, fill=F] (76.63,81.62) circle (0.40mm);
\path[line width=0.30mm, draw=L, fill=F] (78.65,81.62) circle (0.40mm);
\path[line width=0.30mm, draw=L, fill=F] (84.03,57.04) circle (0.40mm);
\path[line width=0.30mm, draw=L, fill=F] (86.22,57.04) circle (0.40mm);
\path[line width=0.30mm, draw=L, fill=F] (88.07,57.21) circle (0.40mm);
\path[line width=0.30mm, draw=L] (65.35,82.12) -- (62.15,57.38);
\path[line width=0.30mm, draw=L] (65.52,81.79) -- (98.17,57.21);
\path[line width=0.30mm, draw=L] (71.91,82.12) -- (73.26,57.38);
\path[line width=0.30mm, draw=L] (72.08,82.29) -- (80.50,57.21);
\path[line width=0.30mm, draw=L] (90.60,57.71) -- (81.68,82.29);
\path[line width=0.30mm, draw=L] (90.43,57.88) -- (89.25,82.46);
\path[line width=0.30mm, draw=L] (97.84,57.54) -- (89.25,82.29);
\draw(63.32,85) node[anchor=base west]{\fontsize{12.23}{15.07}\selectfont $w_{1}$};
\draw(69.70,85) node[anchor=base west]{\fontsize{12.23}{15.07}\selectfont $w_{2}$};
\draw(77.81,85) node[anchor=base west]{\fontsize{12.23}{15.07}\selectfont $w_{l-2}$};
\draw(86.39,85) node[anchor=base west]{\fontsize{12.23}{15.07}\selectfont $w_{l-1}$};
\draw(60.46,52) node[anchor=base west]{\fontsize{12.23}{15.07}\selectfont $u_{1}$};
\draw(66.01,52) node[anchor=base west]{\fontsize{14.23}{17.07}\selectfont $u_{2}$};
\draw(71.73,52) node[anchor=base west]{\fontsize{12.23}{15.07}\selectfont $u_{3}$};
\draw(78.13,52) node[anchor=base west]{\fontsize{12.23}{15.07}\selectfont $u_{4}$};
\draw(88.57,52) node[anchor=base west]{\fontsize{12.23}{15.07}\selectfont $u_{l}$};
\draw(93.31,52) node[anchor=base west]{\fontsize{12.23}{15.07}\selectfont $u_{l+1}$};
\path[line width=0.30mm, draw=L] (158.11,83.35) ellipse (19.11mm and 6.57mm);
\path[line width=0.30mm, draw=L, fill=F] (144.56,82.00) circle (1.00mm);
\path[line width=0.30mm, draw=L, fill=F] (150.79,82.17) circle (1.00mm);
\path[line width=0.30mm, draw=L, fill=F] (160.89,82.17) circle (1.00mm);
\path[line width=0.30mm, draw=L, fill=F] (168.30,82.17) circle (1.00mm);
\path[line width=0.30mm, draw=L] (159.46,55.74) ellipse (22.81mm and 7.07mm);
\path[line width=0.30mm, draw=L, fill=F] (141.37,57.26) circle (1.00mm);
\path[line width=0.30mm, draw=L, fill=F] (147.09,57.43) circle (1.00mm);
\path[line width=0.30mm, draw=L, fill=F] (153.70,57.51) circle (1.00mm);
\path[line width=0.30mm, draw=L, fill=F] (164.52,57.64) circle (1.00mm);
\path[line width=0.30mm, draw=L, fill=F] (169.81,57.43) circle (1.00mm);
\path[line width=0.30mm, draw=L, fill=F] (176.88,57.59) circle (1.00mm);
\path[line width=0.30mm, draw=L] (141.53,57.43) -- (147.09,57.43);
\path[line width=0.30mm, draw=L, fill=F] (153.82,81.83) circle (0.40mm);
\path[line width=0.30mm, draw=L, fill=F] (155.67,81.83) circle (0.40mm);
\path[line width=0.30mm, draw=L, fill=F] (157.69,81.83) circle (0.40mm);
\path[line width=0.30mm, draw=L, fill=F] (159.34,57.42) circle (0.40mm);
\path[line width=0.30mm, draw=L, fill=F] (161.28,57.47) circle (0.40mm);
\path[line width=0.30mm, draw=L, fill=F] (157.29,57.32) circle (0.40mm);
\path[line width=0.30mm, draw=L] (144.40,82.34) -- (141.20,57.59);
\path[line width=0.30mm, draw=L] (144.56,82.00) -- (177.22,57.43);
\path[line width=0.30mm, draw=L] (164.52,57.86) -- (168.30,82.68);
\draw(142.37,85) node[anchor=base west]{\fontsize{12.23}{15.07}\selectfont $w_{1}$};
\draw(148.75,85) node[anchor=base west]{\fontsize{12.23}{15.07}\selectfont $w_{2}$};
\draw(156.86,85) node[anchor=base west]{\fontsize{12.23}{15.07}\selectfont $w_{l-2}$};
\draw(165.44,85) node[anchor=base west]{\fontsize{12.23}{15.07}\selectfont $w_{l-1}$};
\draw(139.50,52) node[anchor=base west]{\fontsize{12.23}{15.07}\selectfont $u_{1}$};
\draw(146.06,52) node[anchor=base west]{\fontsize{12.23}{15.07}\selectfont $u_{2}$};
\draw(151.78,52) node[anchor=base west]{\fontsize{12.23}{15.07}\selectfont $u_{3}$};
\draw(161.90,52.46) node[anchor=base west]{\fontsize{12.23}{15.07}\selectfont $u_{l-1}$};
\draw(169.62,52) node[anchor=base west]{\fontsize{12.23}{15.07}\selectfont $u_{l}$};
\draw(173.94,52) node[anchor=base west]{\fontsize{12.23}{15.07}\selectfont $u_{l+1}$};
\draw(46.5,80.58) node[anchor=base west]{\fontsize{16.92}{20.90}\selectfont $W^{\prime\prime}$};
\draw(47.40,53.13) node[anchor=base west]{\fontsize{16.92}{20.90}\selectfont $U^{\prime}$};
\draw(127.76,80.58) node[anchor=base west]{\fontsize{16.92}{20.90}\selectfont $W^{\prime\prime}$};
\draw(126.40,53.13) node[anchor=base west]{\fontsize{16.92}{20.90}\selectfont $U^{\prime}$};
\path[line width=0.30mm, draw=L] (150.86,82.26) -- (147.00,57.64);
\path[line width=0.30mm, draw=L] (170.07,57.77) -- (176.77,57.77);
\path[line width=0.30mm, draw=L] (150.86,82.13) -- (153.70,57.13);
\path[line width=0.30mm, draw=L] (160.85,82.45) -- (164.52,57.64);
\path[line width=0.30mm, draw=L] (168.62,82.02) -- (169.81,57.42);
\end{tikzpicture}%
\caption{A special $2\ell$-cycle in $G$. }{\label{fig1}}
\end{figure}
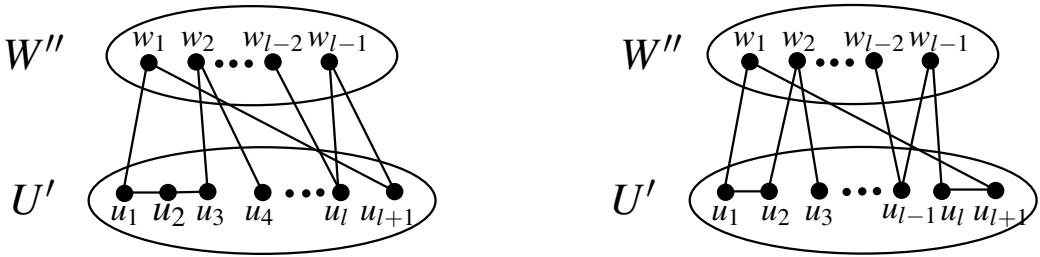

Now we have $V(G)=W''\cup U'$ and $G[W'',U']\cong K_{\lambda,n-\lambda}$.
In the following, we consider two cases of Theorem \ref{theorem1.5}.

(i) $\ell=2$.
Since $|W''|=\lambda=\ell t-1=2t-1$, we can see that
$G[U']$ is $P_3$-free (otherwise, we can find $tC_{4}$ in $G$).
Thus, $G[U']$ consists of independent edges and isolated vertices.
Since $G$ is extremal with respect to $spex(n,tC_{4})$, we know that $G$ is edge-maximal,
which implies that $W''$ is a $(2t-1)$-clique and $G\cong S_{n,2t-1}^{++}$.

(ii) $\ell\geq3$.
Since $|W''|=\lambda=\ell t-1$, we will see that $e(U')\le 1$.
Otherwise, whether $G[U']$ contains a $P_3$ or two independent edges,
we can always find $t$ disjoint copies of $C_{2\ell}$,
which consist of $t-1$ $2\ell$-cycles in $G[W'',U']$,
and a special $2\ell$-cycle (see Figure \ref{fig1}).
Since $G$ is edge-maximal, we similarly have $G\cong S_{n,2t-1}^{+}$.

This completes the proof.
\end{proof}


\begin{thebibliography}{99}
\setlength{\itemsep}{0pt}

\bibitem {Bukh}
B. Bukh, Z.L. Jiang, A bound on the number of edges in graphs without an even cycle, Combin. Probab. Comput. 26 (2017), no. 1, 1--15.


\bibitem {Chen-M}
M.-Z. Chen, A.-M. Liu, X.-D. Zhang, Spectral extremal results with forbidding
linear forests, Graphs Combin. 35 (2019), no. 1, 335--351.

\bibitem {Chen-M2}
M.-Z. Chen, A.-M. Liu, X.-D. Zhang, On the spectral radius of graphs without a star
forest, Discrete Math. 344 (2021), no. 4, Paper No. 112269, 12 pp.


\bibitem {Chvatal}
V. Chv\'{a}tal, D. Hanson, Degrees and matchings, J. Combinatorial Theory Ser. B 20 (1976), no. 2, 128--138.

\bibitem {Cioaba1}
S. Cioab\u{a}, D.N. Desai, M. Tait, The spectral even cycle problem,
arxiv:2205.00990v1 (2022).

\bibitem {Cioaba2}
S. Cioab\u{a}, D.N. Desai, M. Tait, The spectral radius of graphs with no odd wheels,
European J. Combin. 99 (2022), Paper No. 103420, 19 pp.

\bibitem {Cioaba}
S. Cioab\u{a}, L.H. Feng, M. Tait, X.-D. Zhang, The maximum spectral radius of graphs without friendship subgraphs, Electron. J. Combin. 27 (2020), no. 4, Paper No. 4.22, 19 pp.

\bibitem {Desai}
D.N. Desai, L.Y. Kang, Y.T. Li, Z.Y. Ni, M. Tait, J. Wang, Spectral extremal graphs for intersecting cliques, Linear Algebra Appl. 644 (2022), 234--258.

\bibitem {Erdos1}
P. Erd\H{o}s, \"{U}ber ein Extremal problem in der Graphentheorie,
Arch. Math. 13 (1962), 222--227.

\bibitem {Erdos3}
P. Erd\H{o}s, T. Gallai, On maximal paths and circuits of graphs, Acta Math. Acad. Sci.
Hungar. 10 (1959), 337--356.

\bibitem {Erdos2}
P. Erd\H{o}s, L. P\'{o}sa, On the maximal number of disjoint circuits of a graph, Publ. Math. Debrecen 9 (1962), 3--12.

\bibitem {Feng}
L.H. Feng, G.H. Yu, X.-D. Zhang, Spectral radius of graphs with given matching
number, Linear Algebra Appl. 422 (2007), no. 1, 133--138.


\bibitem {Furedi1}
Z. F\"{u}redi, D.S. Gunderson, Extremal numbers for odd cycles,
Combin. Probab. Comput. 24 (2015), no. 4, 641--645.


\bibitem {He}
Z.Y. He, A new upper bound on the Tur\'{a}n number of even cycles, Electron. J. Combin. 28 (2021), no. 2, Paper No. 2.41, 18 pp.


\bibitem {Li}
Y.T. Li, Y.J. Peng, The spectral radius of graphs with no intersecting odd cycles,
Discrete Math. 345 (2022), no. 8, Paper No. 112907, 16 pp.

\bibitem {LIY}
Y.T. Li, Y.J. Peng, The maximum spectral radius of non-bipartite graphs
forbidding short odd cycles, Electron. J. Combin. 29 (2022),
no. 4, Paper No. 4.2.

\bibitem {LIN1}
H.Q. Lin, B. Ning, B. Wu, Eigenvalues and triangles in graphs, Combin. Probab.
Comput. 30 (2021), 258--270.

\bibitem {LIN2}
H.Q. Lin, B. Ning, A complete solution to the Cvetkovi\'{c}-Rowlinson conjecture,
J. Graph Theory 97 (2021), no. 3, 441--450.

\bibitem {Moon}
J.W. Moon, On independent complete subgraphs in a graph, Canadian J. Math. 20
(1968), 95--102.

\bibitem {Ni}
Z.Y. Ni, J. Wang, L.Y. Kang, Spectral extremal graphs for disjoint cliques, Electron. J. Combin. 30 (2023), no. 1, Paper No. 1.20.

\bibitem {Nikiforov5}
V. Nikiforov, Bounds on graph eigenvalues II, Linear Algebra Appl. 427 (2007), no. 2-3, 183--189.

\bibitem {Nikiforov2}
V. Nikiforov, A spectral condition for odd cycles in graphs, Linear Algebra Appl. 428 (2008), no. 7, 1492--1498.


\bibitem {Nikiforov4}
V. Nikiforov, Stability for large forbidden subgraphs,
J. Graph Theory 62 (2009), no. 4, 362--368.

\bibitem {Nikiforov6}
V. Nikiforov, Degree powers in graphs with a forbidden even cycle, Electron. J. Combin. 16 (2009), no. 1, Research Paper 107, 9 pp.

\bibitem {Nikiforov1}
V. Nikiforov, The spectral radius of graphs without paths and cycles of specified length, Linear Algebra Appl. 432 (2010), no. 9, 2243--2256.


\bibitem {NIKI}
V. Nikiforov,
On a theorem of Nosal, arXiv:2104.12171 (2021).

\bibitem {OP}
O. Pikhurko, A note on the Tur\'{a}n function of even cycles, Proc. Amer. Math. Soc. 140
(2012), 3687--3692.

\bibitem {Simonovits}
M. Simonovits, Extremal graph problems with symmetrical extremal graphs.
Additional chromatic conditions, Discrete Math. 7 (1974), 349--376.

\bibitem {TAIT1}
M. Tait, J. Tobin, Three conjectures in extremal spectral graph theory,
J. Comb. Theory, Ser. B 126 (2017), 137--161.

\bibitem {TAIT2}
M. Tait, The Colin de Verdiere parameter, excluded minors, and the spectral radius,
J. Comb. Theory, Ser. A 166 (2019), 42--58.


\bibitem {Verstraete}
J. Verstra\"{e}te, On arithmetic progressions of cycle lengths in graphs,
Combin. Probab. Comput. 9 (2000), no. 4, 369--373.

\bibitem {WANG}
J. Wang, L.Y. Kang, Y.S. Xue, On a conjecture of spectral extremal problems,
J. Combin. Theory Ser. B 159 (2023), 20--41.




\bibitem {Zhai-3}
M.Q. Zhai, B. Wang, Proof of a conjecture on the spectral radius of $C_4$-free graphs,
Linear Algebra Appl. 437 (2012), no. 7, 1641--1647.

\bibitem {ZHAI}
M.Q. Zhai, H.Q. Lin, Spectral extrema of $K_{s,t}$-minor free graphs--on a conjecture of M. Tait, J. Combin. Theory Ser. B 157 (2022), 184--215.
\end{thebibliography}
\end{document}